\documentclass[11pt,reqno]{amsart}

\usepackage[T1]{fontenc}
\usepackage[utf8]{inputenc}
\usepackage[a4paper,margin=29mm]{geometry}
\usepackage{lmodern}
\usepackage{microtype}
\microtypesetup{expansion=false}
\usepackage{amsmath,amssymb,amsthm}
\usepackage{mathtools}
\usepackage{mathrsfs}
\usepackage{enumitem}
\usepackage{tikz-cd}
\usepackage{quiver}
\usepackage{xcolor}
\usepackage{aliascnt}
\usepackage{hyperref}
\usepackage{orcidlink}
\usepackage[nameinlink,capitalize,noabbrev]{cleveref}

\hypersetup{
  colorlinks=true,
  linkcolor=blue!55!black,
  citecolor=green!40!black,
  urlcolor=blue!65!black
}

\numberwithin{equation}{section}

\newtheorem{theorem}{Theorem}[section]

\newaliascnt{proposition}{theorem}
\newtheorem{proposition}[proposition]{Proposition}
\aliascntresetthe{proposition}

\newaliascnt{lemma}{theorem}
\newtheorem{lemma}[lemma]{Lemma}
\aliascntresetthe{lemma}

\newaliascnt{corollary}{theorem}
\newtheorem{corollary}[corollary]{Corollary}
\aliascntresetthe{corollary}

\newaliascnt{assumption}{theorem}

\aliascntresetthe{assumption}

\theoremstyle{definition}

\newaliascnt{definition}{theorem}
\newtheorem{definition}[definition]{Definition}
\aliascntresetthe{definition}

\newaliascnt{example}{theorem}

\aliascntresetthe{example}

\theoremstyle{remark}

\newaliascnt{remark}{theorem}
\newtheorem{remark}[remark]{Remark}
\aliascntresetthe{remark}

\newaliascnt{warning}{theorem}

\aliascntresetthe{warning}

\crefname{proposition}{Proposition}{Propositions}
\Crefname{proposition}{Proposition}{Propositions}
\crefname{lemma}{Lemma}{Lemmas}
\Crefname{lemma}{Lemma}{Lemmas}
\crefname{corollary}{Corollary}{Corollaries}
\Crefname{corollary}{Corollary}{Corollaries}
\crefname{assumption}{Assumption}{Assumptions}
\Crefname{assumption}{Assumption}{Assumptions}
\crefname{definition}{Definition}{Definitions}
\Crefname{definition}{Definition}{Definitions}
\crefname{example}{Example}{Examples}
\Crefname{example}{Example}{Examples}
\crefname{remark}{Remark}{Remarks}
\Crefname{remark}{Remark}{Remarks}
\crefname{warning}{Warning}{Warnings}
\Crefname{warning}{Warning}{Warnings}

\DeclareMathOperator{\age}{age}

\DeclareMathOperator{\Bl}{Bl}

\DeclareMathOperator{\divisor}{div}

\DeclareMathOperator{\Gr}{Gr}

\DeclareMathOperator{\Spec}{Spec}

\newcommand{\C}{\mathbb C}

\newcommand{\Q}{\mathbb Q}

\newcommand{\Z}{\mathbb Z}

\newcommand{\cO}{\mathcal O}

\newcommand{\Firr}{F_{\mathrm{irr}}}

\newcommand{\id}{\mathrm{id}}

 \title[Compactification Independence on Deligne--Mumford Stacks]
{Compactification Independence of the Irregular Hodge Filtration
on Deligne--Mumford Stacks}
\author{Haoxu Wang\,\orcidlink{0009-0004-5598-063X}}
\address{Morningside Center of Mathematics, Academy of Mathematics and
Systems Science, Chinese Academy of Sciences, Beijing 100190, China}
\email{krassotkinkolya@gmail.com}
\subjclass[2020]{Primary 14F40; Secondary 14A20, 14E05, 14J33}
\keywords{irregular Hodge filtration, Deligne--Mumford stacks,
Landau--Ginzburg models, compactification independence, weak factorization}
\hypersetup{
  pdftitle={Compactification Independence of the Irregular Hodge Filtration on Deligne-Mumford Stacks},
  pdfauthor={Haoxu Wang},
  pdfsubject={Irregular Hodge filtration on Deligne-Mumford stack Landau-Ginzburg models},
  pdfkeywords={irregular Hodge filtration, Deligne-Mumford stacks, Landau-Ginzburg models, compactification independence, weak factorization}
}
\date{}
 
\begin{document}

\begin{abstract}
Let \((\mathscr U,w)\) be a smooth separated Deligne--Mumford stack
of finite type over \(\mathbb C\), with a regular function \(w\).
Following Yu, one can use a compactification to define a filtration
on its twisted de Rham cohomology.  The same compactification gives
Kontsevich lattices that compute this filtration.  The extension of
\(w\) on the compactification may be only rational.  Only a local
nondegeneracy condition near the polar divisor is imposed.

We prove that the resulting filtration does not depend on the
compactification.  The proof uses good resolutions and stacky weak
factorization.  This reduces the comparison to blowups and roots
along boundary divisors.  Filtered comparison theorems for the Yu and
Kontsevich complexes are established for both operations.

Harder and Lee use orbifold irregular Hodge numbers in their study of
stacky Clarke mirror pairs.  Compactification independence applies
sector by sector to their setting.  Thus the resulting orbifold
filtration and irregular Hodge numbers depend only on the stack
Landau--Ginzburg model.
\end{abstract}

\maketitle
\tableofcontents

\section{Introduction}
\label{sec:introduction}

\subsection{The compactification problem and mirror symmetry}
\label{subsec:introduction-compactification-problem}

In mirror symmetry, a pair \((Y,v)\), where \(Y\) is a smooth variety
and \(v\) is a regular function on \(Y\), is called a
Landau--Ginzburg model.  Harder--Lee extend this language to smooth
Deligne--Mumford stacks.  Their mirror formula for a
\(d\)-dimensional stacky Clarke mirror pair is
\[
 f_{\mathrm{orb}}^{\lambda,\mu}
 \bigl(\mathcal T(\Sigma),w(\check\Sigma)\bigr)
 =
 f_{\mathrm{orb}}^{d-\lambda,\mu}
 \bigl(\mathcal T(\check\Sigma),w(\Sigma)\bigr)
\]
\cite[Theorem~1.1]{HarderLee}.  The terms
\(f_{\mathrm{orb}}^{\lambda,\mu}\) are orbifold irregular Hodge
numbers.  Their definition starts with twisted de Rham cohomology and
the irregular Hodge filtration.

Let \(\mathscr U\) be a smooth separated
Deligne--Mumford stack of finite type over \(\mathbb C\), and let
\[
 w\colon\mathscr U\longrightarrow\mathbb A^1
\]
be a regular function.  Its twisted de Rham cohomology is
\[
 H^k(\mathscr U,w)
 :=
 H_{\mathrm{dR}}^k(\mathscr U,w)
 :=
 \mathbb H^k\!\left(
  \mathscr U,
  (\Omega_{\mathscr U}^\bullet,d+dw\wedge)
 \right).
\]

Suppose first that \(\mathscr U=U\) is a smooth quasi-projective
variety.  Yu introduced the irregular Hodge filtration on this
cohomology \cite{Yu}.  He chooses a good compactification
\[
 U\hookrightarrow X,
 \qquad
 w\colon X\longrightarrow\mathbb P^1.
\]
Here \(X\) is smooth and proper.  The divisor \(D=X\setminus U\) has
normal crossings, and the extension of \(w\) is a morphism
\cite[\S1(a)]{Yu}.  If \(P=w^*(\infty)\), then, for
\(\lambda\in\mathbb Q\), the level-\(\lambda\) Yu complex has
degree-\(a\) term
\begin{equation*}
 F_{\mathrm{irr}}^\lambda K_{X,w}^a
 =
 \begin{cases}
  0,
  &a<\lceil\lambda\rceil,\\[3pt]
  \Omega_X^a(\log D)
  \bigl(\lfloor(a-\lambda)P\rfloor\bigr),
  &a\geq\lceil\lambda\rceil.
 \end{cases}
\end{equation*}
The induced filtration on cohomology is
\[
 F_{\mathrm{irr},X}^\lambda H^k(U,w)
 :=
 \operatorname{Im}\!\left[
  \mathbb H^k\!\left(
   X,F_{\mathrm{irr}}^\lambda K_{X,w}^\bullet
  \right)
  \longrightarrow H^k(U,w)
 \right].
\]
Yu proves that this subspace does not depend on the good
compactification \(X\).  His proof uses blowups and weak factorization
\cite[Lemma~1.6 and Theorem~1.7]{Yu}.

For a Laurent polynomial \(w\) on
\(T^n=(\mathbb C^*)^n\), Yu also considers smooth toric
compactifications \(T^n\subset X\).  In general, \(w\) extends to
\(X\) only as a rational map to \(\mathbb P^1\).  He proves that the
Newton filtration still computes the irregular Hodge filtration
\cite[\S4]{Yu}.  Related meromorphic constructions are studied in
\cite[\S9]{SabbahYu}.  Chen--Yu later define nondegenerate rational
compactifications and prove compactification independence for smooth
varieties \cite[\S\S2.1--2.3]{ChenYu}.

For Deligne--Mumford stacks, one must also include the twisted
sectors.  Orbifold cohomology is built from the inertia stack.  The contribution
of each sector is shifted by its age \cite{ChenRuan,AGV}.  Let
\[
 e\colon I\mathscr U\longrightarrow\mathscr U,
 \qquad
 I\mathscr U=\coprod_{\nu\in\Lambda}\mathscr U_\nu .
\]
Put \(w_\nu=e^*w|_{\mathscr U_\nu}\), and let
\(a_\nu\in\mathbb Q\) be the age of \(\mathscr U_\nu\).

Harder--Lee use this construction for twisted de Rham cohomology.  On
a fixed nondegenerate projective toroidal orbifold compactification,
they use Yu's filtered de Rham complexes on the twisted sectors.  They
explain the extension from smooth varieties through local quotient
charts \cite[\S3.1 and Remark~2.8]{HarderLee}.  Following
Harder--Lee, we suppress this compactification from the notation.  For
\(q\in\mathbb Q\), they define
\[
 H_{\mathrm{dR,orb}}^q(\mathscr U,w)
 :=
 \bigoplus_{\substack{\nu\in\Lambda\\q-2a_\nu\in\mathbb Z}}
 H_{\mathrm{dR}}^{q-2a_\nu}(\mathscr U_\nu,w_\nu).
\]
Its orbifold filtration is
\[
 F_{\mathrm{orb}}^\lambda
 H_{\mathrm{dR,orb}}^q(\mathscr U,w)
 =
 \bigoplus_{\substack{\nu\in\Lambda\\q-2a_\nu\in\mathbb Z}}
 F_{\mathrm{irr}}^{\lambda-a_\nu}
 H_{\mathrm{dR}}^{q-2a_\nu}(\mathscr U_\nu,w_\nu).
\]
This is the age-shifted orbifold irregular Hodge filtration of
\cite[Definition~3.9]{HarderLee}.
The orbifold irregular Hodge numbers are
\[
 f_{\mathrm{orb}}^{\lambda,\mu}(\mathscr U,w)
 :=
 \dim
 \operatorname{Gr}_{F_{\mathrm{orb}}}^\lambda
 H_{\mathrm{dR,orb}}^{\lambda+\mu}(\mathscr U,w),
 \qquad \lambda,\mu\in\mathbb Q.
\]
Their construction itself does not compare the filtrations obtained
from different compactifications; the independence theorem below shows
a posteriori that the resulting filtration is intrinsic.

\subsection{Main results}
\label{subsec:introduction-main-results}

Our work supplies this global comparison.  We extend
compactification independence from smooth varieties to smooth
Deligne--Mumford stacks.  
Applied to each inertia sector, our
result gives the compactification independence needed for the
Harder--Lee numbers.

We first describe the compactifications used in the main theorem.
Let \(\mathscr X\) be a smooth proper Deligne--Mumford stack that
contains \(\mathscr U\) as an open substack, and let
\[
 \mathscr D=(\mathscr X\setminus\mathscr U)_{\mathrm{red}}.
\]
We assume that \(\mathscr D\) has normal crossings.  We also allow
\(w\) to extend only as a rational function on \(\mathscr X\).  Write
\(\divisor(w)=\mathscr Z-\mathscr P\).  The pole divisor is supported
on \(\mathscr D\).  At each point of \(|\mathscr P|\), we require
\(\mathscr Z\) to be empty or smooth, and
\(\mathscr Z+\mathscr D\) to have normal crossings. We call this an \emph{NC rational stack compactification}.  A
\emph{good stack compactification} is one for which the boundary has
a labeled SNC presentation and \(w\) extends to a morphism
\(\mathscr X\to\mathbb P^1\).  The precise definitions are given in
\cref{def:nc-rational-stack-pair}.

For \(\lambda\in\mathbb Q\), Yu's formula defines a complex
\(F_{\mathrm{irr}}^\lambda
\mathcal K_{\mathscr X,w}^\bullet\) on every compactification of this
kind.  The image of its hypercohomology in \(H^k(\mathscr U,w)\) is
\(F_{\mathrm{irr},\mathscr X}^\lambda H^k(\mathscr U,w)\). See \cref{subsec:yu-filtered-complex}.
The main result is that this subspace does not depend on \(\mathscr X\).

\begin{theorem}\label{thm:introduction-nc-rational-compactification-independence}
Let \((\mathscr U,w)\) be a smooth separated Deligne--Mumford
Landau--Ginzburg model of finite type over \(\mathbb C\).  Let
\(\mathscr X_1\) and \(\mathscr X_2\) be two NC rational stack
compactifications.  Then, for every \(k\in\mathbb Z\) and
\(\lambda\in\mathbb Q\),
\begin{equation*}
 F_{\mathrm{irr},\mathscr X_1}^{\lambda}H^k(\mathscr U,w)
 =
 F_{\mathrm{irr},\mathscr X_2}^{\lambda}H^k(\mathscr U,w)
\end{equation*}
as subspaces of \(H^k(\mathscr U,w)\).
We denote the common subspace by
\[
 F_{\mathrm{irr}}^\lambda H^k(\mathscr U,w).
\]
\end{theorem}

The proof has three ingredients.

\begin{enumerate}[label=\textup{(\arabic*)},leftmargin=*]
\item \emph{Filtered comparison for elementary modifications.}
  Let \(\pi\colon\mathscr Y\to\mathscr X\) be either a root along a
  labeled boundary divisor or the blowup of a smooth
  boundary-admissible center.  Then pullback induces, for every
  \(\lambda\in\mathbb Q\), a quasi-isomorphism
  \[
   F_{\mathrm{irr}}^\lambda\mathcal K_{\mathscr X,w}^\bullet
   \xrightarrow{\ \sim\ }
   R\pi_*F_{\mathrm{irr}}^\lambda
   \mathcal K_{\mathscr Y,w_{\mathscr Y}}^\bullet.
  \]
  See \cref{thm:boundary-root-filtered-invariance}
  and \cref{thm:adapted-boundary-blowup-levelwise-invariance}.  Roots are the
  additional operations required by stacky weak factorization.  The
  blowup theorem simultaneously covers the noncancelling centers of
  Yu \cite[Lemma~1.6]{Yu}, the zero--pole cancellation centers of
  Chen--Yu \cite[\S2.3]{ChenYu}, and the blowups used for boundary
  strictification.

\item \emph{Resolution of NC rational compactifications.}
  Every NC rational stack compactification is dominated by a good
  stack compactification through a representable projective
  birational morphism
  \(\rho\colon\mathscr Y\to\mathscr X\) that is the identity over
  \(\mathscr U\) and is a composite of boundary-admissible blowups;
  see \cref{thm:filtered-good-resolution}.  The construction first
  realizes the barycentric subdivision of the generalized cone
  complex of the boundary by ordinary blowups along smooth boundary
  strata
  \cite[Sections~4.3 and~5.3]{MPS},
  \cite[Sections~3.2.1--3.2.3]{HolmesSchwarz};
  compare \cite[Lemmas~A.2.4 and~A.2.6]{Harper}.
  It then applies the Chen--Yu zero--pole resolution algorithm
  \cite[\S2.3]{ChenYu}.  The comparisons in \textup{(1)} apply at
  every step.

\item \emph{Relative stacky weak factorization.}
  Any two good stack compactifications are joined, relative to
  \(\mathscr U\), by a finite zigzag whose edges, in one orientation,
  are boundary-admissible blowups or roots along labeled boundary
  divisors.  The extensions of \(w\) are compatible along the
  zigzag; see
  \cref{thm:relative-stacky-weak-factorization}.  We use the general
  stacky weak factorization theorem of Bergh--Rydh
  \cite[Theorem~D]{BerghRydh}; compare
  \cite[Theorem~1.1]{Harper} and Bergh's published
  global-quotient case \cite{BerghWF}.  Applying \textup{(1)} after
  the resolutions in \textup{(2)} proves
  \cref{thm:introduction-nc-rational-compactification-independence}.

\end{enumerate}

For completeness, we also prove that good stack compactifications
exist.  The general case uses Rydh's preliminary compactification
theorem \cite[Theorem~F]{RydhCompactification}, whereas the case of
quasi-projective coarse space starts from Kresch's published result
\cite[Theorems~4.4 and~5.3]{KreschGeometry}.  This existence statement
is logically separate from compactification independence; see
\cref{prop:good-stack-compactification-existence}.

\begin{corollary}For any choice of NC rational stack compactifications of the sectors,
the resulting filtration is the filtration \(F_{\mathrm{orb}}\)
above.  Hence
\[
 f_{\mathrm{orb}}^{\lambda,\mu}(\mathscr U,w)
 =
 \dim\operatorname{Gr}_{F_{\mathrm{orb}}}^{\lambda}
 H_{\mathrm{dR,orb}}^{\lambda+\mu}(\mathscr U,w)
\]
depends only on \((\mathscr U,w)\).
\end{corollary}

This follows by applying the main theorem to every sector and taking
the direct sum.

On a fixed nondegenerate projective toroidal orbifold
compactification in the sense of Harder--Lee,
\cref{thm:harder-lee-orbifold-comparison} identifies their filtration
with the sectorwise filtration above.  Thus the orbifold irregular
Hodge numbers in their mirror formula depend only on
\((\mathscr U,w)\), and not on the chosen compactification. 
The numbers in the Harder--Lee formula become
invariants of the Landau--Ginzburg pair.  This gives a basis for their
further use in mirror symmetry.

 \section{Compactified Landau--Ginzburg Models}
\label{sec:compactified-lg-models}

\subsection{Stacks, divisors, and differential forms}
\label{subsec:dm-conventions}

We work over \(\C\).  All schemes and algebraic stacks are of finite
type over \(\C\), and all algebraic stacks are separated
Deligne--Mumford stacks unless stated otherwise.  For such a stack
\(\mathscr X\), we write
\[
 p_{\mathscr X}\colon\mathscr X\longrightarrow X
\]
for its coarse moduli morphism; the coarse space \(X\) is an algebraic
space in general.  A divisor on a smooth stack means a Cartier divisor
on the stack, so its components and multiplicities are measured by
stack valuations rather than by valuations on \(X\).

\begin{definition}[NC, labeled SNC, and SNC divisors on a stack]
\label{def:stack-nc-snc-divisors}
Let \(\mathscr X\) be a smooth Deligne--Mumford stack and let
\(\mathscr D\) be a reduced effective divisor.
\begin{enumerate}[label=\textup{(\roman*)}]
\item The divisor \(\mathscr D\) is \emph{normal crossing} (NC) if,
  at every geometric point, there is an \'etale scheme chart with
  regular parameters \(x_1,\ldots,x_n\) on which
  \[
   \mathscr D=(x_1\cdots x_r=0)
  \]
  for some \(0\leq r\leq n\).
\item A \emph{finite labeled SNC presentation} of \(\mathscr D\) is
  a finite family
  \begin{equation*}
   \boldsymbol{\mathscr D}=\{\mathscr D_i\}_{i\in I}
  \end{equation*}
  of nonempty connected smooth effective Cartier divisors such that
  \(\mathscr D=\bigcup_{i\in I}\mathscr D_i\) as a reduced divisor
  and every intersection \(\bigcap_{j\in J}\mathscr D_j\) is empty
  or smooth of codimension \(|J|\).  The divisor together with this
  chosen presentation is a \emph{labeled SNC divisor}.
\item The divisor \(\mathscr D\) is \emph{SNC} if it admits a finite
  labeled SNC presentation.  Thus SNC is a property, whereas
  labeled SNC includes the presentation as data.
\end{enumerate}
An SNC divisor is NC, but the converse can fail because a branch may
self-intersect or stabilizers may permute local branches.  This is the
stack-theoretic SNC convention of
\cite[Definition~3.1]{BerghRydh}.  An ordering of the finite label set
will be chosen only when required by weak factorization.
\end{definition}

Differentials are taken on the lisse--\'etale site.  We write
\[
 \Omega_{\mathscr X}^a
 :=\bigwedge^a\Omega_{\mathscr X/\C}^1;
\]
these sheaves are locally free when \(\mathscr X\) is smooth.  For an
NC divisor \(\mathscr D\), logarithmic forms are obtained by \'etale
descent from the usual sheaves on scheme charts.  In the coordinates
of \cref{def:stack-nc-snc-divisors},
\(\Omega_{\mathscr X}^1(\log\mathscr D)\) has basis
\[
 \frac{dx_1}{x_1},\ldots,\frac{dx_r}{x_r},
 dx_{r+1},\ldots,dx_n,
 \qquad
 \Omega_{\mathscr X}^a(\log\mathscr D)
 =\bigwedge^a\Omega_{\mathscr X}^1(\log\mathscr D).
\]
If \(\mathscr E\) is an effective divisor supported on
\(\mathscr D\), set
\[
 \mathcal O_{\mathscr X}(*\mathscr E)
 :=\varinjlim_{m\geq0}\mathcal O_{\mathscr X}(m\mathscr E),
 \qquad
 \Omega_{\mathscr X}^a(\log\mathscr D)(*\mathscr E)
 :=\Omega_{\mathscr X}^a(\log\mathscr D)
   \otimes\mathcal O_{\mathscr X}(*\mathscr E).
\]
Thus the latter sheaf permits arbitrary meromorphic poles along
\(|\mathscr E|\), logarithmic poles along the remaining boundary, and
no other poles.  The analogous notation
\(\Omega_{\mathscr X}^a(*\mathscr E)\) omits the logarithmic boundary.
These constructions depend only on \(|\mathscr E|\) and commute with
\'etale pullback.  If
\(j\colon\mathscr X\setminus|\mathscr E|\hookrightarrow\mathscr X\),
then, as may be checked on an \'etale atlas,
\[
 \Omega_{\mathscr X}^a(*\mathscr E)
 \simeq j_*\Omega_{\mathscr X\setminus|\mathscr E|}^a.
\]

In characteristic zero every finite stabilizer is linearly reductive.
Hence the Deligne--Mumford stacks considered here are tame, and
coarse-moduli pushforward is exact on quasi-coherent sheaves by
\cite[Definition~3.1 and Theorem~3.2(a),(b)]{AOV}.  We use coarse
spaces only for the explicitly stated projectivity conditions; the
filtered complexes remain on the stack.

\subsection{NC rational pairs and their compactifications}
\label{subsec:nc-rational-stack-compactifications}

\begin{definition}[Smooth Landau--Ginzburg model]
\label{def:smooth-stack-lg-model}
A \emph{smooth Deligne--Mumford Landau--Ginzburg model} is a pair
\((\mathscr U,w_{\mathscr U})\), where \(\mathscr U\) is smooth and
\(w_{\mathscr U}\in H^0(\mathscr U,\mathcal O_{\mathscr U})\).
\end{definition}

\begin{definition}[NC rational stack models]
\label{def:nc-rational-stack-pair}
An \emph{NC rational stack pair} is a tuple
\[
 (\mathscr X,\mathscr D,\mathscr Z,\mathscr P,w)
\]
with the following properties.
\begin{enumerate}[label=\textup{(\roman*)}]
\item The stack \(\mathscr X\) is smooth.
\item The divisor \(\mathscr D\subset\mathscr X\) is reduced NC.  We
  put
  \[
   \mathscr U_{\mathscr X}:=\mathscr X\setminus|\mathscr D|,
   \qquad
   j_{\mathscr X}\colon\mathscr U_{\mathscr X}\hookrightarrow
   \mathscr X.
  \]
\item The rational map
  \(w\colon\mathscr X\dashrightarrow\mathbb P^1\) restricts to a
  regular function on \(\mathscr U_{\mathscr X}\).  On each component
  where \(w\not\equiv0\),
  \begin{equation*}
   \divisor(w)=\mathscr Z-\mathscr P,
  \end{equation*}
  where \(\mathscr Z\) and \(\mathscr P\) are effective and have no
  common prime component; on a component where \(w=0\), set
  \(\mathscr Z=\mathscr P=0\).  The regularity on
  \(\mathscr U_{\mathscr X}\) is equivalent to
  \(|\mathscr P|\subset|\mathscr D|\).
\item On some neighbourhood \(\mathscr V\) of \(|\mathscr P|\), the
  divisor \(\mathscr Z|_{\mathscr V}\) is empty or smooth and
  \begin{equation*}
   \mathscr Z|_{\mathscr V}+\mathscr D|_{\mathscr V}
  \end{equation*}
  is reduced NC.
\end{enumerate}
Condition~\textup{(iv)} is called \emph{nondegeneracy along the polar
locus}; no condition is imposed on the zero divisor away from that
locus.

The following adjectives and subclasses will be used.
\begin{enumerate}[label=\textup{(\alph*)}]
\item The pair is \emph{strict} if its boundary is equipped with a
  finite labeled SNC presentation.  Thus strictness includes the
  labels as data, not merely their existence.
\item The pair is \emph{morphic} if \(w\) extends to a morphism
  \(\mathscr X\to\mathbb P^1\).
  For a morphic pair, polar nondegeneracy is automatic because the zero
  and pole fibres are disjoint.

\item An \emph{NC rational stack compactification}
  of \((\mathscr U,w_{\mathscr U})\) is an NC rational stack pair with
  \(\mathscr X\) proper, together with an isomorphism
  \[
   (\mathscr U,w_{\mathscr U})
   \xrightarrow{\sim}
   (\mathscr U_{\mathscr X},w|_{\mathscr U_{\mathscr X}}).
  \]
  Equivalently, \(\mathscr U\hookrightarrow\mathscr X\) is dense on
  every component and
  \(\mathscr D=(\mathscr X\setminus\mathscr U)_{\mathrm{red}}\).
  A compactification whose underlying pair is morphic is a
  \emph{morphic NC model}; it is \emph{strict} when the underlying
  pair is strict.
\item A compactification is \emph{projective}
  if the coarse moduli space of \(\mathscr X\) is a projective scheme
  over \(\C\).  This is stronger than properness of the stack.
\item A \emph{good stack compactification}
  is a strict morphic NC rational stack compactification.  A
  \emph{projective good stack compactification}
  is one that is also projective.
\end{enumerate}
\end{definition}

When no confusion is possible, we suppress \(\mathscr Z\) and
\(\mathscr P\) and write an NC rational pair as
\((\mathscr X,\mathscr D,w)\).

The polar nondegeneracy condition is equivalently expressed by the
following Chen--Yu normal form
\cite[\S2.1, Definition and formula~(2), p.~3]{ChenYu}.  At a point of
\(|\mathscr P|\), after passing to a sufficiently small \'etale chart,
there are coordinates \(x_1,\ldots,x_\ell,y_1,\ldots,y_m,z_1,\ldots\)
and positive integers \(e_i\) such that
\begin{equation}
 \mathscr D=(x_1\cdots x_\ell y_1\cdots y_m=0),
 \qquad
 w=\frac{1}{x_1^{e_1}\cdots x_\ell^{e_\ell}}
 \quad\text{or}\quad
 w=\frac{z_1}{x_1^{e_1}\cdots x_\ell^{e_\ell}}.
 \label{eq:chen-yu-nondegenerate-local-form}
\end{equation}
Indeed, the polar branches supply the \(x_i\), while a nonempty zero
divisor supplies one parameter transverse to every boundary branch;
an \'etale root removes the remaining unit.  Conversely, either local
form implies condition~\textup{(iv)}.  In particular,
\(z_1^a/x_1^b\) with \(a>1\) is excluded.

\begin{lemma}\label{lem:rational-map-indeterminacy-locus}
\label{lem:indeterminacy-logarithmic-conormal-line}
Let
\((\mathscr X,\mathscr D,\mathscr Z,\mathscr P,w)\)
be an NC rational stack pair, and set
\[
 I=I(w):=(|\mathscr Z|\cap|\mathscr P|)_{\mathrm{red}}.
\]
\begin{enumerate}[label=\textup{(\roman*)}]
\item The reduced indeterminacy locus of
  \(w\colon\mathscr X\dashrightarrow\mathbb P^1\) is \(I\).
  Consequently, \(w\) is a morphism if and only if \(I=\varnothing\).
\item Suppose the pair is strict.  If \(H\) is a connected labeled
  boundary component occurring in \(\mathscr P\) with coefficient
  \(e>0\), then every connected component \(C\) of
  \(\mathscr Z\times_{\mathscr X}H\) is a smooth closed substack,
  regularly immersed in codimension two and having normal crossings
  with the labeled boundary.  There are finitely many such
  components, and at every point of \(C\) one can choose \'etale
  coordinates in which, up to a unit,
  \begin{equation}
   H=(x_1=0),\qquad C=(z=x_1=0),\qquad
   w=\frac{z}{x_1^e x_2^{e_2}\cdots x_\ell^{e_\ell}}.
   \label{eq:one-step-chen-yu-local-center}
  \end{equation}
\item If \(I\ne\varnothing\), then
  \(I\hookrightarrow\mathscr X\) is a regular immersion of
  codimension two, although \(I\) need not be smooth.  With
  \(\mathcal I_I\) its ideal, define
  \begin{equation*}
   \mathcal T_{I/(\mathscr X,\mathscr D)}^\vee
   :=\operatorname{im}\left(
    \mathcal I_I/\mathcal I_I^2\xrightarrow{d}
    \Omega_{\mathscr X}^1(\log\mathscr D)|_I
   \right).
  \end{equation*}
  On a neighbourhood of \(I\), the map
  \[
   \mathcal O_{\mathscr X}(-\mathscr P)
   \xrightarrow{dw}\Omega_{\mathscr X}^1(\log\mathscr D)
  \]
  is a locally split injection.  If its image is \(\Theta_w\), then
  \begin{equation*}
   \Theta_w|_I
   =\mathcal T_{I/(\mathscr X,\mathscr D)}^\vee
  \end{equation*}
  is a line subbundle.  Moreover, the ordinary conormal line of the
  smooth numerator divisor near \(I\) maps isomorphically onto it:
  \begin{equation*}
   N_{\mathscr Z/\mathscr X}^\vee|_I
   \xrightarrow[\sim]{\ d_{\log}\ }
   \mathcal T_{I/(\mathscr X,\mathscr D)}^\vee.
  \end{equation*}
  These statements commute with \'etale base change.
\end{enumerate}
\end{lemma}

\begin{proof}
For \textup{(i)}, work in a regular local domain \(R\) on an \'etale
chart and write \(w=A/B\), where \(A,B\in R\) have no common
height-one factor.  The map \([A:B]\) is defined away from their
common zero locus.  Conversely, suppose that it extends at a point
where \(A,B\) both vanish, and write the extension as \([a:b]\) with
\((a,b)=R\).  Then \(A=qa\) and \(B=qb\) for some
\(q\in\operatorname{Frac}(R)^\times\).  At each height-one prime, the
minimum of the valuations of each pair is zero, hence \(v(q)=0\).
Normality applied to \(q\) and \(q^{-1}\) makes \(q\) a unit, a
contradiction.

For \textup{(ii)}, the finite-type closed stack
\(\mathscr Z\times_{\mathscr X}H\) has finitely many open-and-closed
connected components.  The second normal form in
\eqref{eq:chen-yu-nondegenerate-local-form} gives
\eqref{eq:one-step-chen-yu-local-center}; it also shows directly that
each component is smooth of codimension two and transverse to every
labeled boundary stratum.  These scheme-theoretic intersections and
their connected components are intrinsic on the stack.

For \textup{(iii)}, the assertion is \'etale-local.  At a point of
\(I\), write
\[
 \mathscr D=(x_1\cdots x_\ell y_1\cdots y_m=0),\qquad
 \mathscr P=\sum_{i=1}^\ell e_i(x_i=0),\qquad
 w=\frac{uz}{g},\qquad
 g=\prod_{i=1}^\ell x_i^{e_i},
\]
where \(u\) is a unit.  Put \(q=x_1\cdots x_\ell\).  Then
\(\mathcal I_I=(z,q)\); this is a regular sequence, although
\(R/(z,q)\) can be singular.  Since
\[
 dq=q\sum_{i=1}^\ell\frac{dx_i}{x_i},
\]
the logarithmic conormal map sends \(\bar z\) to \(dz\) and
\(\bar q\) to zero.  Hence
\(\mathcal T_{I/(\mathscr X,\mathscr D)}^\vee=\mathcal O_I\,dz\).
The same calculation identifies the ordinary conormal line of the
numerator \((z=0)\) with this line.

Finally, \(g\) trivializes
\(\mathcal O_{\mathscr X}(-\mathscr P)\), and
\[
 g\,dw
 =u\,dz+z\,du-uz\sum_{i=1}^\ell e_i\frac{dx_i}{x_i}.
\]
Its \(dz\)-coefficient is a unit near \(I\), so projection to that
coordinate splits the map after shrinking.  Restriction to \(I\)
gives \(\Theta_w|_I=\mathcal O_I\,dz\).  Every construction used is
intrinsic and commutes with \'etale pullback, so the local statements
descend.
\end{proof}

 \section{Yu--Kontsevich Complexes, Orbifold Sectors, and Functorial Comparison}
\label{sec:yu-kontsevich-framework}

Let \((\mathscr U,w)\) be a smooth Deligne--Mumford
Landau--Ginzburg model in the sense of
\cref{def:smooth-stack-lg-model}.  Put
\[
 \nabla_w:=d+dw\wedge.
\]
The exponentially twisted de Rham complex and its cohomology are
\begin{equation*}
 (\Omega_{\mathscr U}^\bullet,\nabla_w),
 \qquad
 H_{\mathrm{dR}}^k(\mathscr U,w)
 :=\mathbb H^k\!\left(
   \mathscr U,\Omega_{\mathscr U}^\bullet,\nabla_w
 \right).
\end{equation*}
For an NC rational stack pair
\((\mathscr X,\mathscr D,\mathscr Z,\mathscr P,w)\), write
\(j\colon\mathscr U_{\mathscr X}\hookrightarrow\mathscr X\) and set
\begin{equation*}
 \mathcal K_{\mathscr X,w}^\bullet
 :=\left(
   \Omega_{\mathscr X}^\bullet(\log\mathscr D)(*\mathscr P),
   \nabla_w
 \right).
\end{equation*}
The sheaf-theoretic constructions below are made on the
lisse--\'etale site.  Comparisons between the logarithmic lattices
can be checked after restriction to an \'etale scheme atlas;
restriction is exact and conservative.  Statements involving
\(R\pi_*\), by contrast, will use proper direct image and flat base
change explicitly.
The terms are quasi-coherent \(\mathcal O\)-modules.  The differential
is only \(\C\)-linear.  Thus complexes are taken in sheaves of
\(\C\)-vector spaces; degreewise direct-image arguments are made in
quasi-coherent modules.

\begin{proposition}\label{prop:stack-logarithmic-comparison}
Restriction to the open model induces a canonical quasi-isomorphism
\begin{equation}
 r_{\mathscr X}\colon
 \mathcal K_{\mathscr X,w}^\bullet
 \xrightarrow{\ \sim\ }
 Rj_*\left(\Omega_{\mathscr U_{\mathscr X}}^\bullet,\nabla_w\right).
 \label{eq:stack-logarithmic-comparison-sheaves}
\end{equation}
Consequently,
\begin{equation}
 R\Gamma(\mathscr X,\mathcal K_{\mathscr X,w}^\bullet)
 \simeq
 R\Gamma\!\left(
  \mathscr U_{\mathscr X},
  \Omega_{\mathscr U_{\mathscr X}}^\bullet,\nabla_w
 \right).
 \label{eq:stack-logarithmic-comparison}
\end{equation}
\end{proposition}

\begin{proof}
Let
\[
 i\colon\mathscr V:=\mathscr X\setminus|\mathscr P|
 \hookrightarrow\mathscr X,
 \qquad
 \mathscr H:=\mathscr D|_{\mathscr V},
 \qquad
 j_{\mathscr V}\colon\mathscr U_{\mathscr X}\hookrightarrow
 \mathscr V.
\]
The function \(w\) is regular on \(\mathscr V\): this may be checked
on a normal affine \'etale chart, where nonnegative valuations at all
height-one primes imply regularity.  Hence
\((\mathcal O_{\mathscr V},d+dw)\) is an integrable logarithmic
connection along \(\mathscr H\).  Its residues are zero.  Indeed, in
an NC coordinate \(x\) the coefficient of \(dx/x\) in \(dw\) is
\(x\partial w/\partial x\), which vanishes on \((x=0)\).

After passing to \'etale scheme charts, Deligne's logarithmic
comparison theorem
\cite[Chapter~II, Corollary~3.14(i)]{DeligneRegular} therefore gives
\begin{equation}
 \left(\Omega_{\mathscr V}^\bullet(\log\mathscr H),\nabla_w\right)
 \xrightarrow{\ \sim\ }
 R(j_{\mathscr V})_*
 \left(\Omega_{\mathscr U_{\mathscr X}}^\bullet,\nabla_w\right).
 \label{eq:deligne-comparison-on-nonpolar-open}
\end{equation}
The local comparisons descend because their canonical cone becomes
acyclic on a surjective \'etale atlas.  Finally, \(i\) is affine and,
on an NC chart, both sides of
\begin{equation*}
 i_*\Omega_{\mathscr V}^a(\log\mathscr H)
 =\Omega_{\mathscr X}^a(\log\mathscr D)(*\mathscr P)
\end{equation*}
are obtained by inverting the local equations of the polar branches.
Thus \(i_*\) is exact on the quasi-coherent terms in question and
transforms \eqref{eq:deligne-comparison-on-nonpolar-open} into
\eqref{eq:stack-logarithmic-comparison-sheaves}.  Applying derived
global sections proves the second assertion.
\end{proof}

\subsection{The Yu filtration}
\label{subsec:yu-filtered-complex}

Fix an NC rational stack pair and write
\(\mathscr P=\sum_i e_i\mathscr P_i\).  For \(c\in\Q\), all
rounding is coefficientwise:
\begin{equation*}
 \lfloor c\mathscr P\rfloor
 :=\sum_i\lfloor ce_i\rfloor\mathscr P_i,
 \qquad
 \lceil c\mathscr P\rceil
 :=\sum_i\lceil ce_i\rceil\mathscr P_i.
\end{equation*}

\begin{definition}[Yu filtered complex]
\label{def:stack-yu-filtered-complex}
For \(\lambda\in\Q\) and \(a\geq0\), define
\begin{equation*}
 \Firr^\lambda\mathcal K_{\mathscr X,w}^a
 :=
 \begin{cases}
  0,&a<\lceil\lambda\rceil,\\[3pt]
  \Omega_{\mathscr X}^a(\log\mathscr D)
   \bigl(\lfloor(a-\lambda)\mathscr P\rfloor\bigr),
   &a\geq\lceil\lambda\rceil.
 \end{cases}
\end{equation*}
\end{definition}

These terms form a subcomplex of
\(\mathcal K_{\mathscr X,w}^\bullet\).  Indeed, logarithmic
differentiation preserves every integral boundary twist, while
\[
 dw\in\Omega_{\mathscr X}^1(\log\mathscr D)(\mathscr P),
 \qquad
 \lfloor(a-\lambda)\mathscr P\rfloor+\mathscr P
 =\lfloor(a+1-\lambda)\mathscr P\rfloor.
\]
The resulting filtration is decreasing and left continuous, and its
jumps are locally discrete.  The definition requires neither
properness nor projectivity and, as a definition of complexes, does
not require polar nondegeneracy.

\begin{definition}[Cohomological Yu filtration]
\label{def:yu-filtration-on-stack-cohomology}
For \(\lambda\in\Q\), the filtration defined by the rational model is
\begin{equation}
 F_{\mathrm{Yu},\mathscr X}^\lambda
 H_{\mathrm{dR}}^k(\mathscr U_{\mathscr X},w)
 :=\operatorname{im}\!\left\{
 \mathbb H^k\!\left(
  \mathscr X,\Firr^\lambda\mathcal K_{\mathscr X,w}^\bullet
 \right)
 \longrightarrow
 H_{\mathrm{dR}}^k(\mathscr U_{\mathscr X},w)
 \right\},
 \label{eq:yu-filtration-on-stack-cohomology}
\end{equation}
where the target is identified by
\cref{prop:stack-logarithmic-comparison}.
\end{definition}

For any decreasing \(\Q\)-indexed filtration, denoted \(F\), on a
vector space or a complex, we use the right-limit convention
\begin{equation*}
 F^{>\lambda}:=\bigcup_{\substack{\lambda'>\lambda\\\lambda'\in\Q}}
 F^{\lambda'},
 \qquad
 \Gr_F^\lambda:=F^\lambda/F^{>\lambda}.
\end{equation*}
Applied to \eqref{eq:yu-filtration-on-stack-cohomology}, this is the
associated graded of the cohomological Yu filtration.  It
must be distinguished from the fixed-fractional integer slice below.

For \(\alpha\in\Q\cap[0,1)\), it is convenient to retain the integer slice
\begin{equation*}
 F_\alpha^p:=\Firr^{p-\alpha},
 \qquad
 \operatorname{gr}_{F_\alpha}^p
 :=\Firr^{p-\alpha}/\Firr^{p+1-\alpha}.
\end{equation*}
All transition comparisons below are made with \(\alpha\) fixed.
If \(\mathscr P\) is reduced, the complex filtration has jumps only
at integral indices and its full level quotient
\(\Gr_{F_{\mathrm{irr}}}^p\mathcal K_{\mathscr X,w}^\bullet\)
equals \(\operatorname{gr}_{F_0}^p\).  For a nonreduced polar
divisor there can be additional jumps strictly between two adjacent
integer-slice levels; in that case
\(\operatorname{gr}_{F_\alpha}^p\) generally aggregates several full
rational level quotients.  These are statements about filtered complexes:
identifying their hypercohomology with the associated graded of the
induced cohomological filtration additionally requires strictness,
for example an appropriate degeneration theorem.
The bounded coherent model is
\begin{equation*}
 \mathcal C_{\mathscr X,w}^\bullet
 :=\Firr^0\mathcal K_{\mathscr X,w}^\bullet
 =\left[
  \mathcal O_{\mathscr X}\longrightarrow
  \Omega_{\mathscr X}^1(\log\mathscr D)(\mathscr P)
  \longrightarrow\cdots\longrightarrow
  \Omega_{\mathscr X}^n(\log\mathscr D)(n\mathscr P)
 \right],
\end{equation*}
with differential \(\nabla_w\), and, for \(\lambda\in\Q\), with
induced filtration
\begin{equation*}
 \Firr^\lambda\mathcal C_{\mathscr X,w}^\bullet
 :=\mathcal C_{\mathscr X,w}^\bullet
   \cap\Firr^\lambda\mathcal K_{\mathscr X,w}^\bullet
 =\begin{cases}
   \mathcal C_{\mathscr X,w}^\bullet,&\lambda\leq0,\\
   \Firr^\lambda\mathcal K_{\mathscr X,w}^\bullet,&\lambda>0.
  \end{cases}
\end{equation*}

Below, all filtration indices are rational.  In particular, the
fractional index \(\alpha\) always lies in \(\Q\cap[0,1)\).

\subsection{Kontsevich lattices and the local comparison}
\label{subsec:stack-kontsevich-lattices}

The rounded Yu lattices need not be preserved by pullback when a zero
cancels part of a pole.  Kontsevich lattices retain precisely the
additional \(dw\)-direction needed to control this phenomenon.
Their Hodge-theoretic interpretation through mixed twistor
\(D\)-modules is developed in \cite{MochizukiKontsevich}.

\begin{definition}[Stack Kontsevich lattices]
\label{def:stack-kontsevich-lattice}
For \(\alpha\in\Q\cap[0,1)\) and \(a\geq0\), set
\begin{equation}
 \Omega_{\mathscr X,w}^a(\alpha)
 :=\ker\!\left\{
  \Omega_{\mathscr X}^a(\log\mathscr D)
    (\lfloor\alpha\mathscr P\rfloor)
  \xrightarrow{\ \nabla_w\ }
  \frac{\Omega_{\mathscr X}^{a+1}(*\mathscr D)}
  {\Omega_{\mathscr X}^{a+1}(\log\mathscr D)
    (\lfloor\alpha\mathscr P\rfloor)}
 \right\}.
 \label{eq:stack-kontsevich-lattice}
\end{equation}
The kernel is an \(\mathcal O_{\mathscr X}\)-submodule: the extra
term \(dh\wedge\omega\) in \(\nabla_w(h\omega)\) has the same
logarithmic pole order as \(\omega\).  Since \(\nabla_w^2=0\), the
lattices form the bounded Kontsevich complex
\[
 \bigl(\Omega_{\mathscr X,w}^\bullet(\alpha),\nabla_w\bigr).
\]
For \(p\in\Z\), write
\begin{equation*}
 \mathcal A_{\mathscr X,w}^{p,\alpha}
 :=\sigma_{\geq p}
 \bigl(\Omega_{\mathscr X,w}^\bullet(\alpha),\nabla_w\bigr),
\end{equation*}
where \(\sigma_{\geq p}\) is stupid truncation.
\end{definition}

\begin{proposition}[Local Kontsevich--Yu comparison]
\label{prop:local-kontsevich-yu-package}
\label{lemma:intersection-of-logarithmic-and-polar-forms}
Let \((\mathscr X,\mathscr D,\mathscr Z,\mathscr P,w)\) be an NC
rational stack pair and let \(\alpha\in\Q\cap[0,1)\).
\begin{enumerate}[label=\textup{(\roman*)}]
\item On a sufficiently small neighbourhood \(\mathscr V\) of
  \(|\mathscr P|\), the morphism
  \[
   \mathcal O_{\mathscr V}(-\mathscr P)
   \xrightarrow{\,dw\,}
   \Omega_{\mathscr V}^1(\log\mathscr D)
  \]
  is a locally split injection.  If \(M\) is its cokernel, then, for
  every \(a\geq1\), there are locally split exact sequences
  \begin{align*}
   0&\longrightarrow
   dw\wedge\Omega_{\mathscr V}^{a-1}(\log\mathscr D)(-\mathscr P)
   \longrightarrow\Omega_{\mathscr V}^a(\log\mathscr D)
   \longrightarrow\bigwedge^aM\longrightarrow0,\\
   0&\longrightarrow
   dw\wedge\Omega_{\mathscr V}^{a-1}(*\mathscr D)
   \longrightarrow\Omega_{\mathscr V}^a(*\mathscr D)
   \longrightarrow(\bigwedge^aM)(*\mathscr D)\longrightarrow0.
  \end{align*}
  Inside \(\Omega_{\mathscr V}^a(*\mathscr D)\),
  \begin{equation}
   \bigl(dw\wedge\Omega_{\mathscr V}^{a-1}(*\mathscr D)\bigr)
   \cap\Omega_{\mathscr V}^a(\log\mathscr D)
   =dw\wedge
    \Omega_{\mathscr V}^{a-1}(\log\mathscr D)(-\mathscr P).
   \label{eq:primitive-logarithmic-intersection}
  \end{equation}
  Moreover, \(\Omega_{\mathscr X,w}^a(\alpha)\) is locally free of
  rank \(\binom{\dim\mathscr X}{a}\), and its formation commutes with
  \'etale base change.  With
  \(\Omega_{\mathscr X}^{-1}(\log\mathscr D)=0\), one has
  \begin{equation}
   \Omega_{\mathscr X,w}^a(0)
   =\Omega_{\mathscr X}^a(\log\mathscr D)(-\mathscr P)
    +dw\wedge
     \Omega_{\mathscr X}^{a-1}(\log\mathscr D)(-\mathscr P)
   \label{eq:zero-level-kontsevich-lattice}
  \end{equation}
  inside \(\Omega_{\mathscr X}^a(*\mathscr D)\), and
  \begin{equation}
   \Omega_{\mathscr X,w}^a(\alpha)
   \simeq\Omega_{\mathscr X,w}^a(0)
   \otimes\mathcal O_{\mathscr X}(\lfloor\alpha\mathscr P\rfloor).
   \label{eq:kontsevich-lattice-fractional-twist}
  \end{equation}

\item For every \(p\in\Z\), inclusion in the meromorphic complex
  induces a natural quasi-isomorphism
  \begin{equation*}
   \kappa_{\mathscr X}^{p,\alpha}\colon
   \mathcal A_{\mathscr X,w}^{p,\alpha}
   \xrightarrow{\ \sim\ }
   \Firr^{p-\alpha}\mathcal K_{\mathscr X,w}^\bullet.
  \end{equation*}

\item The untruncated inclusion is a quasi-isomorphism
  \begin{equation*}
   \kappa_{\mathscr X}^{\alpha}\colon
   \bigl(\Omega_{\mathscr X,w}^\bullet(\alpha),\nabla_w\bigr)
   \xrightarrow{\ \sim\ }
   \mathcal K_{\mathscr X,w}^\bullet.
  \end{equation*}

\item The inclusion
  \begin{equation*}
   (\mathcal C_{\mathscr X,w}^\bullet,\Firr)
   \longrightarrow
   (\mathcal K_{\mathscr X,w}^\bullet,\Firr)
  \end{equation*}
  is a filtered quasi-isomorphism.
\end{enumerate}
All these maps commute with \'etale base change; those in
\textup{(ii)} and \textup{(iii)} also commute with restriction to
the open model.
\end{proposition}

\begin{proof}
Away from \(|\mathscr P|\), the function \(w\) is regular and
\[
 \Omega_{\mathscr X,w}^a(\alpha)
 =\Omega_{\mathscr X}^a(\log\mathscr D),
\]
so the lattice assertions in \textup{(i)} are immediate there.  Near
the polar locus, work on an \'etale chart carrying the normal form
\eqref{eq:chen-yu-nondegenerate-local-form}, and write
\[
 x^e:=\prod_i x_i^{e_i},
 \qquad
 x^{\lfloor\alpha e\rfloor}
 :=\prod_i x_i^{\lfloor\alpha e_i\rfloor},
 \qquad
 \vartheta:=x^e dw.
\]
In the two possible normal forms,
\[
 \vartheta=-\sum_i e_i\frac{dx_i}{x_i},
 \qquad\text{or}\qquad
 \vartheta=dz-z\sum_i e_i\frac{dx_i}{x_i}.
\]
Thus \(\vartheta\) is primitive in
\(\Omega^1(\log\mathscr D)\) and may be completed to a logarithmic
basis.  Choosing the complementary bundle \(M\), one obtains
\begin{equation*}
 \Omega^a(\log\mathscr D)
 =\vartheta\wedge\bigwedge^{a-1}M\oplus\bigwedge^aM.
\end{equation*}
This proves the split injection and the first exterior-power sequence
in \textup{(i)}.  After localizing along \(\mathscr D\), the identity
\(dw=x^{-e}\vartheta\) gives the second sequence.  The direct-sum
decompositions before and after localization also give
\eqref{eq:primitive-logarithmic-intersection}.

The kernel in \eqref{eq:stack-kontsevich-lattice} is freely generated
by
\begin{equation*}
 x^{-\lfloor\alpha e\rfloor}\vartheta\wedge\xi,
 \qquad
 x^{e-\lfloor\alpha e\rfloor}\eta,
 \qquad
 \xi\in\bigwedge^{a-1}M,\quad
 \eta\in\bigwedge^aM.
\end{equation*}
This proves local freeness, \eqref{eq:zero-level-kontsevich-lattice},
and \eqref{eq:kontsevich-lattice-fractional-twist}.  Flat pullback
commutes with the defining kernel, so the construction is compatible
with \'etale base change.

For \textup{(ii)}, the assertion is \'etale local.  On a scheme chart
with the normal form
\eqref{eq:chen-yu-nondegenerate-local-form}, it is the third
quasi-isomorphism of \cite[Proposition~1(ii)]{ChenYu}.  The proof
there begins by noting that the statement is local for coherent
sheaves and then uses only these normal forms.  Thus neither
properness nor projectivity is used.  Since the cone is acyclic on a
surjective \'etale atlas, it is acyclic on \(\mathscr X\).

For \textup{(iii)}, take any \(p\leq0\).  Then
\[
 \mathcal A_{\mathscr X,w}^{p,\alpha}
 =\bigl(\Omega_{\mathscr X,w}^\bullet(\alpha),\nabla_w\bigr).
\]
By \textup{(ii)}, its inclusion into
\(\Firr^{p-\alpha}\mathcal K_{\mathscr X,w}^\bullet\) is a
quasi-isomorphism.  These maps are compatible as \(p\) decreases,
and
\[
 \bigcup_{p\leq0}
 \Firr^{p-\alpha}\mathcal K_{\mathscr X,w}^\bullet
 =\mathcal K_{\mathscr X,w}^\bullet.
\]
Filtered colimits are exact.  Taking the union proves
\textup{(iii)}.

For \textup{(iv)}, the two filtered complexes agree at every positive
level.  At a nonpositive level \(\lambda\in\Q\), put
\(p=\lceil\lambda\rceil\) and \(\alpha=p-\lambda\).  Then
\(p\leq0\), so \textup{(ii)} and \textup{(iii)} show that the natural
inclusion
\[
 \Firr^\lambda\mathcal K_{\mathscr X,w}^\bullet
 \longrightarrow\mathcal K_{\mathscr X,w}^\bullet
\]
is a quasi-isomorphism.  The same argument at \(\lambda=0\) shows
that
\(\mathcal C_{\mathscr X,w}^\bullet
 =\Firr^0\mathcal K_{\mathscr X,w}^\bullet\)
maps quasi-isomorphically to the total complex.  The inclusion from
the coherent level to the level \(\lambda\) is therefore a
quasi-isomorphism by two-out-of-three.

All lattices and arrows used above are intrinsic and commute with
\'etale base change.  This completes the proof.
\end{proof}

\subsection{Kontsevich pullback and comparison roofs}
\label{subsec:kontsevich-functorial-comparison}

\begin{definition}[Morphisms of NC rational stack pairs]
\label{def:nc-rational-stack-pair-morphism}
Let
\((\mathscr X',\mathscr D',\mathscr Z',\mathscr P',w')\) and
\((\mathscr X,\mathscr D,\mathscr Z,\mathscr P,w)\) be NC rational
stack pairs.  A \emph{morphism of NC rational stack pairs} is a
commutative diagram
\begin{equation*}
 \begin{tikzcd}
  \mathscr U' \arrow[r,hook] \arrow[d,"g"']
  &\mathscr X' \arrow[d,"\pi"]\\
  \mathscr U \arrow[r,hook]
  &\mathscr X
 \end{tikzcd}
\end{equation*}
such that \(\pi^{-1}|\mathscr D|\subseteq|\mathscr D'|\) and
\(w'=\pi^*w\) as rational functions.  It is \emph{dominant} if
\(\pi\) is dominant on every connected component.  It is an
\emph{NC rational stack-pair modification} if \(\mathscr U'=\mathscr U\),
\(g\) is the identity, and \(\pi\) is proper.
\end{definition}

\begin{proposition}\label{prop:stack-kontsevich-functoriality}
Every dominant morphism of NC rational stack pairs induces injective
\(\mathcal O_{\mathscr X'}\)-linear maps
\begin{equation}
 \theta_{\pi,\alpha}^a\colon
 \pi^*\Omega_{\mathscr X,w}^a(\alpha)
 \lhook\joinrel\longrightarrow
 \Omega_{\mathscr X',w'}^a(\alpha)
 \label{eq:stack-kontsevich-functorial-pullback}
\end{equation}
for all \(a\geq0\) and \(\alpha\in\Q\cap[0,1)\).  Pullback of sections
commutes with the twisted differentials and gives a morphism
\begin{equation}
 \pi^{-1}\Omega_{\mathscr X,w}^\bullet(\alpha)
 \longrightarrow
 \Omega_{\mathscr X',w'}^\bullet(\alpha)
 \label{eq:stack-kontsevich-complex-pullback}
\end{equation}
of complexes of \(\C\)-vector-space sheaves.  These maps are stable
under composition.
\end{proposition}

\begin{proof}
Let
\[
 \widetilde\Omega_{\mathscr X,w}^a
 :=\Omega_{\mathscr X,w}^a(0)(\mathscr P).
\]
By \eqref{eq:zero-level-kontsevich-lattice},
\begin{equation*}
 \widetilde\Omega_{\mathscr X,w}^a
 =\Omega_{\mathscr X}^a(\log\mathscr D)
  +dw\wedge\Omega_{\mathscr X}^{a-1}(\log\mathscr D).
\end{equation*}
The boundary condition on \(\pi\) gives pullback of logarithmic
forms, and \(dw'=\pi^*(dw)\).  Hence rational pullback restricts to
\begin{equation}
 \pi^*\widetilde\Omega_{\mathscr X,w}^a
 \longrightarrow\widetilde\Omega_{\mathscr X',w'}^a.
 \label{eq:unrounded-kontsevich-pullback}
\end{equation}

Put \(\beta:=1-\alpha\).  Since the coefficients of \(\mathscr P\)
are integral, \eqref{eq:kontsevich-lattice-fractional-twist} becomes
\begin{equation}
 \Omega_{\mathscr X,w}^a(\alpha)
 =\widetilde\Omega_{\mathscr X,w}^a
   (-\lceil\beta\mathscr P\rceil).
 \label{eq:fractional-lattice-from-unrounded}
\end{equation}
The polar divisor of \(\pi^*w\) satisfies
\(\mathscr P'\leq\pi^*\mathscr P\).  If
\(\mathscr P=\sum_i e_i\mathscr P_i\), monotonicity and
subadditivity of coefficientwise ceilings give
\begin{equation}
 \pi^*\lceil\beta\mathscr P\rceil
 =\sum_i\lceil\beta e_i\rceil\pi^*\mathscr P_i
 \geq\lceil\beta\pi^*\mathscr P\rceil
 \geq\lceil\beta\mathscr P'\rceil.
 \label{eq:rounded-polar-pullback-inequality}
\end{equation}
Twisting \eqref{eq:unrounded-kontsevich-pullback} accordingly yields
the required factorization
\begin{align*}
 \pi^*\Omega_{\mathscr X,w}^a(\alpha)
 &\simeq
 \pi^*\widetilde\Omega_{\mathscr X,w}^a
 \otimes\mathcal O_{\mathscr X'}
       (-\pi^*\lceil\beta\mathscr P\rceil)\\
 &\longrightarrow
 \widetilde\Omega_{\mathscr X',w'}^a
       (-\pi^*\lceil\beta\mathscr P\rceil)\\
 &\lhook\joinrel\longrightarrow
 \widetilde\Omega_{\mathscr X',w'}^a
       (-\lceil\beta\mathscr P'\rceil)
 \simeq\Omega_{\mathscr X',w'}^a(\alpha).
\end{align*}
Thus ordinary pullback of rational forms factors through
\eqref{eq:stack-kontsevich-functorial-pullback}.  It is generically
injective because \(\pi\) is dominant and the characteristic is zero;
its locally free source then makes it injective everywhere.  Finally,
ordinary pullback commutes with \(d\) and with \(dw\wedge\), and
composes strictly.  This proves the remaining assertions.
\end{proof}

\begin{definition}[Kontsevich-induced Yu comparison]
\label{def:kontsevich-induced-yu-comparison}
Let \(\pi\) be a dominant morphism and let \(\lambda\in\Q\).  Put
\[
 p=\lceil\lambda\rceil,
 \qquad
 \alpha=p-\lambda\in\Q\cap[0,1).
\]
The
\emph{Kontsevich-induced Yu comparison} is the morphism represented
by the roof
\begin{equation}
 \begin{tikzcd}[column sep=large]
 &\mathcal A_{\mathscr X,w}^{p,\alpha}
  \arrow[dl,"\kappa_{\mathscr X}^{p,\alpha}"',"\sim"]
  \arrow[r]
 &R\pi_*\mathcal A_{\mathscr X',w'}^{p,\alpha}
  \arrow[dr,"R\pi_*\kappa_{\mathscr X'}^{p,\alpha}","\sim"']&\\
 \Firr^\lambda\mathcal K_{\mathscr X,w}^\bullet
 &&&R\pi_*\Firr^\lambda\mathcal K_{\mathscr X',w'}^\bullet.
 \end{tikzcd}
 \label{eq:kontsevich-induced-yu-comparison-roof}
\end{equation}
It is denoted by
\begin{equation}
 \Phi_{\pi,\lambda}\colon
 \Firr^\lambda\mathcal K_{\mathscr X,w}^\bullet
 \longrightarrow
 R\pi_*\Firr^\lambda\mathcal K_{\mathscr X',w'}^\bullet.
 \label{eq:kontsevich-induced-yu-comparison}
\end{equation}
The middle arrow is adjoint to the truncation of
\eqref{eq:stack-kontsevich-complex-pullback}.  The roof is needed
because direct pullback need not preserve the rounded Yu lattices.
\end{definition}

\begin{remark}\label{rem:rounded-yu-pullback-obstruction}
In degree \(a\), direct pullback of the level \(\lambda\) Yu lattice
would require
\[
 \pi^*\lfloor(a-\lambda)\mathscr P\rfloor
 \leq\lfloor(a-\lambda)\mathscr P'\rfloor,
\]
which does not follow from \(\mathscr P'\leq\pi^*\mathscr P\) and
can fail after zero--pole cancellation.  Formula
\eqref{eq:fractional-lattice-from-unrounded} replaces this by the
oppositely directed ceiling inequality
\eqref{eq:rounded-polar-pullback-inequality}.
\end{remark}

Let
\begin{equation*}
 \beta_{\pi,\alpha}\colon
 \bigl(\Omega_{\mathscr X,w}^\bullet(\alpha),\nabla_w\bigr)
 \longrightarrow
 R\pi_*\bigl(
  \Omega_{\mathscr X',w'}^\bullet(\alpha),\nabla_{w'}
 \bigr)
\end{equation*}
be adjoint to the untruncated Kontsevich pullback.  To define the
corresponding total map, write
\(j\colon\mathscr U\hookrightarrow\mathscr X\) and
\(j'\colon\mathscr U'\hookrightarrow\mathscr X'\).  Ordinary
pullback on the open stacks, followed by adjunction and the canonical
identification \(Rj_*Rg_*\simeq R\pi_*Rj'_*\), gives
\[
 \gamma_g\colon
 Rj_*(\Omega_{\mathscr U}^\bullet,\nabla_w)
 \longrightarrow
 R\pi_*Rj'_*(\Omega_{\mathscr U'}^\bullet,\nabla_{w'}).
\]
Using \cref{prop:stack-logarithmic-comparison}, define
\begin{equation*}
 \Psi_\pi
 :=(R\pi_*r_{\mathscr X'})^{-1}\circ\gamma_g\circ r_{\mathscr X}
 \colon
 \mathcal K_{\mathscr X,w}^\bullet
 \longrightarrow R\pi_*\mathcal K_{\mathscr X',w'}^\bullet.
\end{equation*}

\begin{proposition}\label{prop:kontsevich-comparison-functoriality}
\label{prop:yu-total-comparison-compatibility}
Let \(\pi\) be a dominant morphism of NC rational stack pairs and
let \(\lambda\in\Q\).
\begin{enumerate}[label=\textup{(\roman*)}]
\item For every composable dominant morphism
  \(\rho\colon\mathscr X''\to\mathscr X'\), one has
  \begin{align*}
   \Phi_{\pi\circ\rho,\lambda}
   &=R\pi_*(\Phi_{\rho,\lambda})\circ\Phi_{\pi,\lambda},
   \\
   \Psi_{\pi\circ\rho}
   &=R\pi_*(\Psi_\rho)\circ\Psi_\pi
  \end{align*}
  under the standard composition identifications for derived direct
  images.

\item For fixed \(\alpha\in\Q\cap[0,1)\) and integers \(p\leq q\), the
  comparisons commute with the integer-slice transition:
  \begin{equation*}
   \begin{tikzcd}[column sep=large]
    \Firr^{q-\alpha}\mathcal K_{\mathscr X,w}^\bullet
     \arrow[r]\arrow[d,"\Phi_{\pi,q-\alpha}"']
    &\Firr^{p-\alpha}\mathcal K_{\mathscr X,w}^\bullet
     \arrow[d,"\Phi_{\pi,p-\alpha}"]\\
    R\pi_*\Firr^{q-\alpha}\mathcal K_{\mathscr X',w'}^\bullet
     \arrow[r]
    &R\pi_*\Firr^{p-\alpha}\mathcal K_{\mathscr X',w'}^\bullet.
   \end{tikzcd}
  \end{equation*}

\item The square
  \begin{equation}
   \begin{tikzcd}[column sep=large]
    \Firr^\lambda\mathcal K_{\mathscr X,w}^\bullet
     \arrow[r,"\Phi_{\pi,\lambda}"]
     \arrow[d]
    &R\pi_*\Firr^\lambda\mathcal K_{\mathscr X',w'}^\bullet
     \arrow[d]\\
    \mathcal K_{\mathscr X,w}^\bullet
     \arrow[r,"\Psi_\pi"']
    &R\pi_*\mathcal K_{\mathscr X',w'}^\bullet
   \end{tikzcd}
   \label{eq:yu-total-compatibility-square}
  \end{equation}
  commutes in the derived category.

\item On the open stacks both \(\Phi_{\pi,\lambda}\) and
  \(\Psi_\pi\) are induced by ordinary pullback.  Consequently,
  \begin{equation*}
   g^*F_{\mathrm{Yu},\mathscr X}^{\lambda}
   H_{\mathrm{dR}}^k(\mathscr U,w)
   \subseteq
   F_{\mathrm{Yu},\mathscr X'}^{\lambda}
   H_{\mathrm{dR}}^k(\mathscr U',w').
  \end{equation*}
\end{enumerate}
\end{proposition}

\begin{proof}
Rational pullback, truncation, and the inclusions
\(\kappa^{p,\alpha}\) compose naturally.  Adjunction and derived
direct image are compatible with composition, proving the first
identity in \textup{(i)}.  The second follows in the same way from
the functoriality of open pullback and the definition of \(\Psi\).
The inclusions
\(\sigma_{\geq q}\Omega_w^\bullet(\alpha)\to
  \sigma_{\geq p}\Omega_w^\bullet(\alpha)\)
commute strictly with Kontsevich pullback, which gives \textup{(ii)}.

For \textup{(iii)}, naturality of restriction and adjunction gives
\begin{equation*}
 R\pi_*r_{\mathscr X'}\circ
 R\pi_*\kappa_{\mathscr X'}^\alpha\circ\beta_{\pi,\alpha}
 =\gamma_g\circ r_{\mathscr X}\circ\kappa_{\mathscr X}^\alpha.
\end{equation*}
Thus the total morphism induced by untruncated Kontsevich pullback is
\(\Psi_\pi\), independently of \(\alpha\).  Put
\(p=\lceil\lambda\rceil\) and \(\alpha=p-\lambda\).  The truncation
inclusion
\[
 \mathcal A_{\mathscr X,w}^{p,\alpha}
 \longrightarrow\Omega_{\mathscr X,w}^\bullet(\alpha)
\]
commutes with pullback, and its composites with
\(\kappa_{\mathscr X}^{p,\alpha}\) and
\(\kappa_{\mathscr X}^{\alpha}\) are the natural inclusions into
\(\mathcal K_{\mathscr X,w}^\bullet\).  Hence the two paths in
\eqref{eq:yu-total-compatibility-square} agree after precomposition
with the quasi-isomorphism \(\kappa_{\mathscr X}^{p,\alpha}\), so
the square commutes.  Restricting this construction to the opens
gives ordinary pullback, and taking images after derived global
sections proves \textup{(iv)}.
\end{proof}

\begin{corollary}\label{cor:levelwise-images-from-total-comparison}
Let \(\lambda\in\Q\).  Suppose that \(\pi\) is a modification and that
\(\Phi_{\pi,\lambda}\) is a quasi-isomorphism.  Under the common-open
identification of twisted de Rham cohomology,
\[
 F_{\mathrm{Yu},\mathscr X}^{\lambda}
 H_{\mathrm{dR}}^k(\mathscr U,w)
 =F_{\mathrm{Yu},\mathscr X'}^{\lambda}
 H_{\mathrm{dR}}^k(\mathscr U,w)
\]
for every \(k\).
\end{corollary}

\begin{proof}
Apply derived global sections to
\eqref{eq:yu-total-compatibility-square}.  The lower map becomes the
identity under the common-open identification, and the upper map is
an isomorphism.  Their image subspaces are therefore equal.
\end{proof}

\begin{remark}\label{rem:fractional-transition-obstruction}
The comparisons are compatible with all integer transitions for a
fixed \(\alpha\), because they arise from truncations of one
Kontsevich complex.  No compatibility between distinct fractional
parameters is asserted.  None is needed for the associated graded
pieces used below.  Those pieces are formed only after all level
images have been identified inside the same total cohomology group.
\end{remark}

\subsection{Acyclic exceptional defects}
\label{subsec:acyclic-kontsevich-defects}

The preceding constructions separate the formal comparison from its
geometric input.  For the modifications considered later, that input
is the vanishing of the following degreewise cokernels.

\begin{proposition}\label{prop:acyclic-defect-descent-criterion}
Let
\[
 \pi\colon
 (\mathscr X',\mathscr D',\mathscr Z',\mathscr P',w')
 \longrightarrow
 (\mathscr X,\mathscr D,\mathscr Z,\mathscr P,w)
\]
be a proper dominant morphism of NC rational stack pairs such that
\begin{equation}
 \mathcal O_{\mathscr X}\xrightarrow{\ \sim\ }
 R\pi_*\mathcal O_{\mathscr X'}.
 \label{eq:kontsevich-structure-sheaf-descent}
\end{equation}
For \(a\geq0\) and \(\alpha\in\Q\cap[0,1)\), put
\[
 Q_\alpha^a:=\operatorname{coker}(\theta_{\pi,\alpha}^a).
\]
If \(R\pi_*Q_\alpha^a=0\) for every \(a\), then
\begin{align}
 \Omega_{\mathscr X,w}^a(\alpha)
 &\xrightarrow{\ \sim\ }
 R\pi_*\Omega_{\mathscr X',w'}^a(\alpha),
 \label{eq:abstract-kontsevich-termwise-descent}\\
 \mathcal A_{\mathscr X,w}^{p,\alpha}
 &\xrightarrow{\ \sim\ }
 R\pi_*\mathcal A_{\mathscr X',w'}^{p,\alpha}
 \label{eq:abstract-kontsevich-truncated-descent}
\end{align}
for every \(p\), and
\(\Phi_{\pi,p-\alpha}\) is a quasi-isomorphism.

If the vanishing holds for every \(\alpha\in\Q\cap[0,1)\), then the
map \(\Phi_{\pi,\lambda}\) is a quasi-isomorphism for every
\(\lambda\in\Q\).  For each fixed \(\alpha\), the comparisons also
identify the integer-filtered diagrams
\(F_\alpha^p=\Firr^{p-\alpha}\), their successive quotients, and the
corresponding filtrations on the coherent Yu models.  Their filtered
spectral sequences agree from the \(E_1\)-page onward.  The levelwise
comparisons are compatible with the maps to the total twisted de
Rham complex; for a modification, the resulting image subspaces in
\(H_{\mathrm{dR}}^k(\mathscr U,w)\) are equal.
\end{proposition}

\begin{proof}
The maps \(\theta_{\pi,\alpha}^a\) are injective by
\cref{prop:stack-kontsevich-functoriality}.  Apply \(R\pi_*\) to
\[
 0\longrightarrow\pi^*\Omega_{\mathscr X,w}^a(\alpha)
 \longrightarrow\Omega_{\mathscr X',w'}^a(\alpha)
 \longrightarrow Q_\alpha^a\longrightarrow0.
\]
Since the Kontsevich lattice is locally free, the projection formula
and \eqref{eq:kontsevich-structure-sheaf-descent} identify the
derived direct image of the first term with
\(\Omega_{\mathscr X,w}^a(\alpha)\).  The assumed vanishing of the
third term then gives
\eqref{eq:abstract-kontsevich-termwise-descent}.

Filter the bounded truncated Kontsevich complex by stupid degree.
The hyperdirect-image spectral sequence has
\[
 E_1^{a,b}=R^b\pi_*
 \Omega_{\mathscr X',w'}^a(\alpha)
\]
and, by the termwise comparison, only the row \(b=0\).  This proves
\eqref{eq:abstract-kontsevich-truncated-descent}; inserting it into
the roof \eqref{eq:kontsevich-induced-yu-comparison-roof} proves the
levelwise Yu quasi-isomorphism.

For fixed \(\alpha\), compatibility with the transition maps is
\cref{prop:kontsevich-comparison-functoriality}\textup{(ii)}.  The
map \(\beta_{\pi,\alpha}\) preserves the stupid filtration on the
untruncated Kontsevich complexes, and its associated graded maps are
the termwise quasi-isomorphisms above; hence it identifies the
spectral sequences from \(E_1\).  Taking cones gives the comparison
of successive integer-slice quotients, and
\cref{prop:local-kontsevich-yu-package}\textup{(iv)} transports the
same conclusion to the coherent Yu models.  Finally, total
compatibility and equality of images follow from
\cref{prop:yu-total-comparison-compatibility} and
\cref{cor:levelwise-images-from-total-comparison}.
\end{proof}

Thus the ordinary-blowup and root-comparison arguments below need
only verify \eqref{eq:kontsevich-structure-sheaf-descent} and the
vanishing of the appropriate exceptional defects \(Q_\alpha^a\).

\subsection{Orbifold sectors and the Harder--Lee filtration}
\label{subsec:orbifold-irregular-hodge-filtration}

The construction so far is the twisted de Rham theory of the stack
itself.  Stabilizers enter through descent of differential forms, but
no additional summands are contributed by nonidentity automorphisms.
To compare with the orbifold theory of Harder--Lee, one must apply the
preceding construction to every component of the inertia stack and
then insert the usual age shifts.

Let
\[
 e\colon I\mathscr U\longrightarrow\mathscr U,
 \qquad
 I\mathscr U=\coprod_{\nu\in\Lambda}\mathscr U_\nu
\]
be the inertia stack and its decomposition into connected
components.  The set \(\Lambda\) is finite.  If \((x,g)\) is a
geometric point of \(\mathscr U_\nu\), write the eigenvalues of \(g\)
on \(T_x\mathscr U\) as
\(\exp(2\pi i m_j/r)\), where \(g\) has order \(r\) and
\(0\leq m_j<r\).  The number
\begin{equation*}
 a_\nu:=\age(\mathscr U_\nu)
 :=\sum_j\frac{m_j}{r}\in\Q
\end{equation*}
is locally constant and hence constant on \(\mathscr U_\nu\).  Put
\(w_\nu=e^*w|_{\mathscr U_\nu}\); see also
\cite[\S4]{BCS} for the inertia-sector and age conventions.

Each \(\mathscr U_\nu\) is a smooth separated Deligne--Mumford stack
and hence admits a good compactification by
\cref{prop:good-stack-compactification-existence}.

\begin{definition}[Modelwise orbifold twisted de Rham filtration]
\label{def:orbifold-irregular-hodge-filtration}
 A \emph{sector compactification datum}
 \(\boldsymbol{\mathscr X}=\{\mathscr X_\nu\}_{\nu\in\Lambda}\)
 consists of an NC rational stack compactification
 \[
 (\mathscr X_\nu,\mathscr D_\nu,\mathscr Z_\nu,
    \mathscr P_\nu,w_\nu)
 \]
 of each \((\mathscr U_\nu,w_\nu)\).  For \(q\in\Q\), define the
age-shifted orbifold twisted de Rham
cohomology by
\begin{equation}
 H_{\mathrm{dR,orb}}^q(\mathscr U,w)
 :=\bigoplus_{\substack{\nu\in\Lambda\\q-2a_\nu\in\Z}}
 H_{\mathrm{dR}}^{q-2a_\nu}(\mathscr U_\nu,w_\nu).
 \label{eq:orbifold-twisted-de-rham-cohomology}
\end{equation}
The datum \(\boldsymbol{\mathscr X}\) gives it the decreasing
modelwise filtration
\begin{equation}
 F_{\mathrm{orb},\boldsymbol{\mathscr X}}^\lambda
 H_{\mathrm{dR,orb}}^q(\mathscr U,w)
 :=\bigoplus_{\substack{\nu\in\Lambda\\q-2a_\nu\in\Z}}
 F_{\mathrm{Yu},\mathscr X_\nu}^{\lambda-a_\nu}
 H_{\mathrm{dR}}^{q-2a_\nu}(\mathscr U_\nu,w_\nu),
 \qquad \lambda\in\Q.
 \label{eq:orbifold-irregular-hodge-filtration}
\end{equation}
No product on \eqref{eq:orbifold-twisted-de-rham-cohomology} is
asserted.  For \(\lambda,\mu\in\Q\), set
\begin{equation*}
 f_{\mathrm{orb},\boldsymbol{\mathscr X}}^{\lambda,\mu}
 (\mathscr U,w)
 :=\dim\Gr_{F_{\mathrm{orb},\boldsymbol{\mathscr X}}}^{\lambda}
 H_{\mathrm{dR,orb}}^{\lambda+\mu}(\mathscr U,w).
\end{equation*}
\end{definition}

The orbifolds considered by Harder--Lee are specifically
\emph{toroidal orbifolds}: analytically they have polydisc charts
modulo finite diagonal groups
\cite[Definition~2.6]{HarderLee}.  Their assertion immediately before
\cite[Definition~3.9]{HarderLee} that every twisted sector is again
nondegenerate is valid for the local normal-crossing structure needed
by the Yu complex, after taking the reduced boundary and the actual
zero and pole divisors.  It should not be read as preserving the
global irreducibility requirement on the zero divisor in
\cite[Definition~3.2]{HarderLee}: its restriction to a sector can
split or be empty, and the restricted potential can be identically
zero.  Our NC rational condition deliberately allows these cases and
retains exactly the local zero--pole geometry used in their
construction.

For comparison with a compactified inertia stack, the following
condition records the only additional geometric point.

\begin{definition}[Inertia-admissible compactification]
\label{def:inertia-admissible-compactification}
Let
\((\mathscr X,\mathscr D,\mathscr Z,\mathscr P,w)\) be an NC
rational stack compactification of \((\mathscr U,w)\).  It is
\emph{inertia-admissible} if, for every connected component
\(\mathscr V\subset I\mathscr X\) for which
\(\mathscr V^\circ:=\mathscr V\times_{\mathscr X}\mathscr U\) is
nonempty, the data
\[
 \mathscr D_{\mathscr V}:=(\mathscr V\setminus
 \mathscr V^\circ)_{\mathrm{red}},
 \qquad
 w_{\mathscr V}:=e^*w|_{\mathscr V},
\]
together with the actual relatively prime zero and pole divisors of
\(w_{\mathscr V}\), form an NC rational stack compactification of
\((\mathscr V^\circ,w_{\mathscr V}|_{\mathscr V^\circ})\).  In
the case \(w_{\mathscr V}\equiv0\), both divisors are set equal to
zero as in \cref{def:nc-rational-stack-pair}.  In
particular, the polar divisor is
\(\divisor_\infty(w_{\mathscr V})\), not the formal restriction of
\(\mathscr P\); the latter can be larger if restriction causes
zero--pole cancellation.
\end{definition}

\begin{lemma}\label{lem:diagonal-inertia-admissibility}
\begin{enumerate}[label=\textup{(\roman*)}]
\item A strict NC rational stack compactification is
  inertia-admissible.
\item Let \((X,D,w)\) be a nondegenerate projective toroidal orbifold
  Landau--Ginzburg compactification in the sense of
  \cite[Definitions~2.6, 3.1, and~3.2]{HarderLee}, and let
  \((\mathscr X,\mathscr D,w)\) be its associated stack
  compactification.  Then it is inertia-admissible in the NC rational
  sense of \cref{def:inertia-admissible-compactification}.  No
  irreducibility is asserted for the zero divisor on a sector, and a
  sector on which the restricted potential is zero is assigned zero
  actual zero and pole divisors.
\end{enumerate}
\end{lemma}

\begin{proof}
For a separated Deligne--Mumford stack the inertia morphism is finite.
Moreover, fixed loci of finite-order automorphisms of a smooth
characteristic-zero chart are smooth.  Hence every component
\(\mathscr V\subset I\mathscr X\) is smooth and proper; only the
boundary and the restricted potential require verification.

For \textup{(i)}, write the chosen labeled boundary as
\(\boldsymbol{\mathscr D}=\{\mathscr D_i\}_{i\in I}\).  At a point
\((x,g)\in\mathscr V\cap\mathscr D_i\), the automorphism \(g\)
preserves \(\mathscr D_i\) and acts on its normal line.  A nontrivial
normal character would make the fixed locus locally lie in
\(\mathscr D_i\); then \(\mathscr V\cap\mathscr D_i\) would contain
an open subset of \(\mathscr V\).  Since \(\mathscr V\) is smooth and
connected, hence irreducible, this would give
\(\mathscr V\subset\mathscr D_i\).  This contradicts
\(\mathscr V^\circ\ne\varnothing\).  Thus every nonempty restriction
\(\mathscr D_i|_{\mathscr V}\) is transverse, and these restrictions
form a labeled SNC divisor.

At a point over \(|\mathscr P|\), use
\eqref{eq:chen-yu-nondegenerate-local-form}.  All denominator
characters are trivial by the preceding paragraph.  In the form with
a numerator, invariance of \(w\) makes its conormal character the
product of the denominator characters with their multiplicities, so
it too is trivial.  The normal form therefore restricts to
\(\mathscr V\).  If instead \(w|_{\mathscr V}\equiv0\), then
\(\mathscr V\) cannot meet \(|\mathscr P|\), by the same normal-form
argument; its actual zero and pole divisors are both zero by
convention.  This proves \textup{(i)}.

For \textup{(ii)}, work in a Harder--Lee uniformizing chart
\(\Delta^n/G\), with \(G\) finite diagonal, and fix \(g\in G\).
The local component of the sector is the fixed locus \((\Delta^n)^g\),
modulo the centralizer of \(g\).  In normal-crossing coordinates
adapted to the boundary, a boundary coordinate of nontrivial
\(g\)-weight vanishes on \((\Delta^n)^g\).  Such a component is
contained in \(D\) and does not meet the open model.  On every
component that does meet the open model, the remaining boundary
coordinates restrict to coordinates on the fixed locus and give the
reduced NC boundary.

The same chart is adapted to the local zero divisor.  Its defining
coordinate either restricts to a parameter transverse to the
boundary, or vanishes identically on the fixed component.  In the
first case the restricted rational function retains the local NC
zero--pole form, after combining coincident restricted polar branches
in its actual polar divisor.  In the second case it is identically
zero, and both actual divisors are set to zero.  Thus the reduced
boundary and the actual relatively prime zero and pole divisors give
an NC rational sector model.  The construction is invariant under
the centralizer of \(g\), hence descends from the uniformizing chart.
\end{proof}

\begin{theorem}[Harder--Lee comparison]
\label{thm:harder-lee-orbifold-comparison}
Let \((X,D,w)\) be a nondegenerate projective toroidal orbifold
Landau--Ginzburg compactification in the sense of Harder--Lee
\cite[Definitions~2.6, 3.1, and~3.2]{HarderLee}.
Let \((\mathscr X,\mathscr D,w)\) be its associated Deligne--Mumford
compactification, and put
\[
 \mathscr U:=\mathscr X\setminus\mathscr D,
\]
and let \(\boldsymbol{\mathscr X}\) be the sector compactification
datum supplied by
\cref{lem:diagonal-inertia-admissibility}\textup{(ii)}.  Then there is
a canonical isomorphism of \(\Q\)-graded filtered vector spaces
\begin{equation}
 \bigl(H_{\mathrm{orb}}^*(X\setminus D,w),F_{\mathrm{irr}}\bigr)
 \xrightarrow{\ \sim\ }
 \bigl(H_{\mathrm{dR,orb}}^*(\mathscr U,w),
       F_{\mathrm{orb},\boldsymbol{\mathscr X}}\bigr).
 \label{eq:harder-lee-stack-orbifold-comparison}
\end{equation}
\end{theorem}

\begin{proof}
Fix a connected component
\(\mathscr X_\nu\subset I\mathscr X\) meeting \(\mathscr U\), and let
\(X_\nu\) be the corresponding analytic orbifold sector.  Write
\[
 \rho_\nu\colon\mathscr X_\nu^{\mathrm{an}}\longrightarrow X_\nu,
 \qquad
 \mathcal C_\nu^\bullet
 :=\Firr^0\mathcal K_{\mathscr X_\nu,w_\nu}^\bullet .
\]
On a quotient chart \([V/G]\), analytic coarse pushforward amounts to
taking \(G\)-invariant sections.  Since finite-group invariants are
exact in characteristic zero, the chartwise identifications of
logarithmic forms with prescribed pole bounds glue and give
\[
 \mathcal C_{\nu,\mathrm{orb}}^\bullet
 =\rho_{\nu,*}\bigl((\mathcal C_\nu^\bullet)^{\mathrm{an}}\bigr),
 \qquad
 F_{\mathrm{irr}}^\lambda\mathcal C_{\nu,\mathrm{orb}}^\bullet
 =\rho_{\nu,*}\bigl(
   (\Firr^\lambda\mathcal C_\nu^\bullet)^{\mathrm{an}}\bigr).
\]
By \cref{lem:diagonal-inertia-admissibility}\textup{(ii)},
\(\mathscr P_\nu=\divisor_\infty(w_\nu)\); hence these are precisely
the sectorwise Yu complex and filtration of Harder--Lee
\cite[\S\S2.3, 3.1, and Definition~3.9]{HarderLee}, including the
convention at nonpositive levels.

Every term of these bounded complexes is coherent on the proper stack
\(\mathscr X_\nu\).  Proper stack GAGA
\cite[Theorem~7.1]{PortaYuGAGA}, applied termwise, identifies the
\(E_1\)-pages of the algebraic and analytic hypercohomology spectral
sequences, both for the total complex and for every filtered level.
It therefore gives compatible isomorphisms on hypercohomology.
Using
\cref{prop:local-kontsevich-yu-package}\textup{(iv),
prop:stack-logarithmic-comparison} and taking the images of the
filtered-level maps identifies Harder--Lee's sector filtration with
\(F_{\mathrm{Yu},\mathscr X_\nu}\).

Finally, Harder--Lee's shifts of cohomological degree and filtration
index are respectively \(2a_\nu\) and \(a_\nu\), exactly as in
\eqref{eq:orbifold-twisted-de-rham-cohomology} and
\eqref{eq:orbifold-irregular-hodge-filtration}.  Passing to connected
sector components and taking their finite direct sum proves
\eqref{eq:harder-lee-stack-orbifold-comparison}.
\end{proof}

\begin{remark}[Open inertia versus compactified inertia]
\label{rem:open-inertia-not-compactified-inertia}
Definition \ref{def:orbifold-irregular-hodge-filtration} uses
\(I\mathscr U\), not all components of a chosen \(I\mathscr X\).
A root construction supported on the boundary can create new inertia
components contained entirely in that boundary, but it is the
identity on \(\mathscr U\); such components have empty open sector
and contribute nothing.  Conversely, changing the stabilizer
structure on \(\mathscr U\) changes its inertia sectors and is not a
compactification modification.  The orbifold filtration should, and
in general does, detect that change.
\end{remark}
 \section{Elementary Stacky Modifications and Filtered Invariance}
\label{sec:elementary-stacky-modifications}

The preceding section reduced filtered invariance to the vanishing of
the Kontsevich lattice defects under derived pushforward.  We now
verify that criterion for the two elementary operations occurring in
relative stacky weak factorization: roots of labeled boundary
components and ordinary blowups in smooth boundary-compatible
centers.  Although motivated by good models, the calculations are
formulated for NC rational pairs whenever the same local geometry is
available.

\begin{definition}[Elementary admissible modification]
\label{def:elementary-admissible-modification}
Let
\((\mathscr X,\mathscr D,\mathscr Z,\mathscr P,w)\) be an NC
rational stack pair.  An \emph{elementary admissible modification}
is one of the following morphisms, equipped with the reduced
inverse-image boundary, the pullback potential, and the transformed
zero and pole divisors specified below:
\begin{enumerate}[label=\textup{(\alph*)}]
\item the blowup of a boundary-admissible center in the sense of
  \cref{def:boundary-admissible-center}; or
\item an \(r\)-th root construction, \(r\geq2\), along a smooth
  effective Cartier component \(\mathscr H\) of \(\mathscr D\).
\end{enumerate}
The first is called a \emph{boundary-admissible ordinary blowup},
and the second a \emph{boundary root modification}.
\end{definition}

In the SNC setting these are the smooth stacky blowups of Harper and
Bergh--Rydh
\cite[Definition~2.2.2]{Harper}
\cite[\S\S3.4--3.6]{BerghRydh}.

\subsection{Root modifications and their filtered comparison}
\label{subsec:root-modifications-filtered-comparison}

Although boundary roots will later be applied to good stack
compactifications, neither their local geometry nor their filtered
comparison requires properness, morphicity of \(w\), or a labeled
SNC boundary.

The algebro-geometric properties of root stacks used below are
standard: we use Cadman's root and quotient descriptions
\cite[\S\S2.2--2.4]{Cadman}, tame coarse pushforward from
\cite[Definition~3.1 and Theorem~3.2]{AOV}, logarithmic pullback from
\cite[Lemma~3.5]{BorneLaaroussi}, and the root transform of standard
pairs from \cite[\S3.5]{BerghRydh}.  The filtered argument combines
these facts with the Kontsevich--Yu lattice description of
\cite{Yu,ESY,ChenYu}; its only new inputs are the floor identity and
the nontrivial-character vanishing made explicit below.

Let \(\mathscr X\) be a smooth
Deligne--Mumford stack over \(\C\), let
\(\mathscr H\subset\mathscr X\) be a smooth effective Cartier
divisor, and let \(r\geq2\).  Write
\begin{equation*}
 \pi\colon
 \mathscr Y:=\sqrt[r]{(\mathscr X,\mathscr H)}
 \longrightarrow \mathscr X
\end{equation*}
for the \(r\)-th root stack.
Denote the universal root line bundle and section by
\((\mathscr L,\sigma)\), and put
\[
 \mathscr H_{\mathscr Y}:=(\sigma=0).
\]

\begin{proposition}\label{prop:boundary-root-basic-properties}
In the situation above, the following assertions hold.
\begin{enumerate}[label=\textup{(\roman*)}]
\item
Let \(V=\Spec A\to\mathscr X\) be an \'etale scheme chart on which
\(\cO_{\mathscr X}(\mathscr H)\) is trivial and
\(\mathscr H|_V=(h=0)\).  Put
\begin{equation*}
 B:=A[t]/(t^r-h),
 \qquad
 \zeta\cdot t:=\zeta t\quad(\zeta\in\mu_r),
\end{equation*}
where \(\mu_r\) acts trivially on \(A\).  There is a Cartesian
diagram
\begin{equation}
 \begin{tikzcd}
 {[\Spec B/\mu_r]} \ar[r] \ar[d] &
 \mathscr Y \ar[d,"\pi"]\\
 \Spec A \ar[r] & \mathscr X .
 \end{tikzcd}
 \label{diag:root-local-chart}
\end{equation}
Under this identification, \(\mathscr H_{\mathscr Y}\) is cut out
by \(t\).  In particular,
\begin{equation}
 \pi^*\mathscr H=r\mathscr H_{\mathscr Y}
 \label{eq:root-divisor-pullback}
\end{equation}
as Cartier divisors.
\item
The stack \(\mathscr Y\) is a smooth Deligne--Mumford stack, and
\(\pi\) is proper and an isomorphism over
\(\mathscr X\setminus\mathscr H\).  It is a relative
coarse-moduli morphism, in the sense that after every \'etale base
change from a scheme chart of \(\mathscr X\), the induced map is the
coarse-moduli morphism; on \eqref{diag:root-local-chart}, its target
is \(\Spec(B^{\mu_r})=\Spec A\).  Direct image is exact on
quasi-coherent sheaves, and for every quasi-coherent
\(\cO_{\mathscr Y}\)-module \(\mathscr E\),
\begin{equation*}
 R^q\pi_*\mathscr E=0\quad(q>0),
 \qquad
 \pi_*\cO_{\mathscr Y}=\cO_{\mathscr X}.
\end{equation*}
\item
Let \(\mathscr D\) be a reduced NC divisor on \(\mathscr X\), and
assume that \(\mathscr H\) is a component of \(\mathscr D\).  Then
\[
 \mathscr D_{\mathscr Y}
 :=(\pi^{-1}\mathscr D)_{\mathrm{red}}
\]
is NC and contains \(\mathscr H_{\mathscr Y}\) as a component.
Every component of \(\mathscr D\) different from
\(\mathscr H\) has pullback multiplicity one at its generic point,
whereas \(\mathscr H\) has pullback multiplicity \(r\).  If
\(\boldsymbol{\mathscr D}=\{\mathscr D_i\}_{i\in I}\) is finite
labeled SNC and \(\mathscr H=\mathscr D_h\), then the reduced
pullbacks of the \(\mathscr D_i\), with \(\mathscr D_h\) replaced by
\(\mathscr H_{\mathscr Y}\), form a finite labeled SNC boundary.
\end{enumerate}
\end{proposition}

\begin{proof}
All assertions are \'etale local on \(\mathscr X\).  Root stacks
commute with base change \cite[Remark~2.2.3]{Cadman}, and Cadman's
quotient description \cite[Example~2.4.1]{Cadman} gives
\eqref{diag:root-local-chart}.  If \(h\) is chosen as a local
parameter, then \(B\) is smooth with parameters
\(t,x_2,\ldots,x_n\).  The Deligne--Mumford and complement assertions
are \cite[Theorem~2.3.3 and Example~2.4.2]{Cadman}; after a scheme-chart
base change, \(\pi\) is the coarse-moduli morphism by
\cite[Corollary~2.3.7]{Cadman}, hence is proper.

On the quotient chart, direct image is the invariants functor on
\(\mu_r\)-equivariant \(B\)-modules.  Since \(\mu_r/\C\) is linearly
reductive, this functor is exact
\cite[Definition~3.1 and Theorem~3.2(a),(b)]{AOV}; moreover
\(B^{\mu_r}=A\).  \'Etale descent therefore gives \textup{(ii)}.

Finally, in coordinates
\(\mathscr H=(x_1=0)\) and
\(\mathscr D=(x_1\cdots x_\ell=0)\), the root chart has
\(x_1=t^r\) and reduced boundary
\((t x_2\cdots x_\ell=0)\).  This proves \textup{(i)} and
\textup{(iii)}, including the multiplicities and the labeled SNC
assertion; compare the root transform of standard pairs in
\cite[\S3.5]{BerghRydh}.
\end{proof}

\begin{lemma}\label{lem:boundary-root-pullback-rounding}
Let
\((\mathscr X,\mathscr D,\mathscr Z,\mathscr P,w)\)
be an NC rational stack pair, and let \(\mathscr H\) be a smooth
effective Cartier component of \(\mathscr D\).  Let
\[
 \pi\colon\mathscr Y=\sqrt[r]{(\mathscr X,\mathscr H)}
 \longrightarrow\mathscr X
\]
be the \(r\)-th root, and put
\[
 \mathscr D_{\mathscr Y}:=(\pi^{-1}\mathscr D)_{\mathrm{red}},
 \qquad
 w_{\mathscr Y}:=\pi^*w.
\]
With the convention of \cref{def:nc-rational-stack-pair} on
components where \(w\) is identically zero, one has
\begin{equation}
 \divisor(w_{\mathscr Y})
 =\pi^*\divisor(w)
 =\pi^*\mathscr Z-\pi^*\mathscr P.
 \label{eq:boundary-root-rational-divisor-pullback}
\end{equation}
In particular,
\[
 \mathscr Z_{\mathscr Y}:=\pi^*\mathscr Z,
 \qquad
 \mathscr P_{\mathscr Y}:=\pi^*\mathscr P
\]
are the zero and pole divisors of \(w_{\mathscr Y}\), and
\[
 (\mathscr Y,\mathscr D_{\mathscr Y},
  \mathscr Z_{\mathscr Y},\mathscr P_{\mathscr Y},w_{\mathscr Y})
\]
is an NC rational stack pair.

Write
\(\mathscr H=\coprod_\nu\mathscr H_\nu\), put
\(\mathscr H_{\mathscr Y,\nu}
 :=(\pi^{-1}\mathscr H_\nu)_{\mathrm{red}}\), and let
\[
 e_\nu:=\operatorname{ord}_{\mathscr H_\nu}(\mathscr P).
\]
For \(\alpha\in\Q\cap[0,1)\), define
\begin{equation*}
 \delta_{\alpha,\nu}
 :=\lfloor\alpha r e_\nu\rfloor
   -r\lfloor\alpha e_\nu\rfloor
 =\lfloor r\{\alpha e_\nu\}\rfloor,
 \qquad
 \Delta_\alpha
 :=\sum_\nu\delta_{\alpha,\nu}
 \mathscr H_{\mathscr Y,\nu}.
\end{equation*}
Then \(0\leq\delta_{\alpha,\nu}\leq r-1\), and
\begin{equation}
 \lfloor\alpha\mathscr P_{\mathscr Y}\rfloor
 =\pi^*\lfloor\alpha\mathscr P\rfloor+\Delta_\alpha.
 \label{eq:boundary-root-rounded-divisor}
\end{equation}
Equivalently,
\begin{equation*}
 \mathcal O_{\mathscr Y}
 \bigl(\lfloor\alpha\mathscr P_{\mathscr Y}\rfloor
       -\mathscr P_{\mathscr Y}\bigr)
 \simeq
 \pi^*\mathcal O_{\mathscr X}
 \bigl(\lfloor\alpha\mathscr P\rfloor-\mathscr P\bigr)
 \otimes\mathcal O_{\mathscr Y}(\Delta_\alpha).
\end{equation*}
\end{lemma}

\begin{proof}
On every component where \(w\) is not identically zero, pullback of
its principal divisor gives
\eqref{eq:boundary-root-rational-divisor-pullback}.  The stated
convention handles the remaining components.  The two
effective pullbacks have no common prime component.  Thus they are
the actual zero and pole divisors of \(w_{\mathscr Y}\).

It remains to check polar nondegeneracy.  This is \'etale local and
is unchanged away from \(\mathscr H\).  At a point of
\(|\mathscr P|\cap\mathscr H\), choose coordinates as in
\eqref{eq:chen-yu-nondegenerate-local-form}; this is the local form
of \cite[\S2.1, formula~(2), pp.~2--3]{ChenYu}.  If
\(\mathscr H=(x_i=0)\) is polar, the root chart replaces
\(x_i\) by \(t^r\), so the polar exponent \(e_i\) becomes
\(re_i\).  If \(\mathscr H=(y_j=0)\) is nonpolar, the same chart
replaces \(y_j\) by \(t^r\) and leaves the displayed expression for
\(w\) unchanged.  In both cases the numerator, when present, remains
a transverse parameter.  The reduced boundary remains NC by
\cref{prop:boundary-root-basic-properties}\textup{(iii)}.  Hence the
transformed data form an NC rational stack pair.

By \eqref{eq:root-divisor-pullback}, the coefficient of
\(\mathscr P_{\mathscr Y}\) along
\(\mathscr H_{\mathscr Y,\nu}\) is \(re_\nu\).  Therefore the
coefficient of the difference in
\eqref{eq:boundary-root-rounded-divisor} is
\[
 \lfloor\alpha r e_\nu\rfloor
 -r\lfloor\alpha e_\nu\rfloor
 =\lfloor r\{\alpha e_\nu\}\rfloor.
\]
It lies between \(0\) and \(r-1\).  Every polar component different
from \(\mathscr H\) has pullback multiplicity one at its generic
point by
\cref{prop:boundary-root-basic-properties}\textup{(iii)}.  This proves
\eqref{eq:boundary-root-rounded-divisor}.  The line-bundle identity
follows because
\(\mathscr P_{\mathscr Y}=\pi^*\mathscr P\).
\end{proof}

\begin{theorem}[Boundary-root invariance]
\label{thm:boundary-root-filtered-invariance}
Let
\((\mathscr X,\mathscr D,\mathscr Z,\mathscr P,w)\) be an NC
rational stack pair, and let \(\mathscr H\) be a smooth effective
Cartier component of \(\mathscr D\).  Put
\[
 \pi\colon\mathscr Y=\sqrt[r]{(\mathscr X,\mathscr H)}
 \longrightarrow\mathscr X,\qquad
 \mathscr D_{\mathscr Y}:=(\pi^{-1}\mathscr D)_{\mathrm{red}},
 \qquad w_{\mathscr Y}:=\pi^*w,
\]
and let
\[
 \mathscr Z_{\mathscr Y}:=\pi^*\mathscr Z,
 \qquad
 \mathscr P_{\mathscr Y}:=\pi^*\mathscr P.
\]
Then
\[
 (\mathscr Y,\mathscr D_{\mathscr Y},
  \mathscr Z_{\mathscr Y},\mathscr P_{\mathscr Y},w_{\mathscr Y})
\]
is an NC rational stack pair.  The induced map
\[
 \mathscr U_{\mathscr Y}
 \xrightarrow{\ \sim\ }
 \mathscr U_{\mathscr X}
\]
is the identity under the natural identification, \(\pi\) is an NC
rational stack-pair modification, and
\begin{align}
 \mathcal O_{\mathscr X}
 &\xrightarrow{\ \sim\ }
 R\pi_*\mathcal O_{\mathscr Y}
 \label{eq:boundary-root-structure-sheaf-descent}\\
 \Omega_{\mathscr X,w}^a(\alpha)
 &\xrightarrow{\sim}
 R\pi_*\Omega_{\mathscr Y,w_{\mathscr Y}}^a(\alpha),
 \label{eq:boundary-root-kontsevich-termwise-descent}\\
 \Phi_{\pi,\lambda}\colon
 \Firr^\lambda\mathcal K_{\mathscr X,w}^\bullet
 &\xrightarrow{\sim}
 R\pi_*\Firr^\lambda
 \mathcal K_{\mathscr Y,w_{\mathscr Y}}^\bullet
 \label{eq:boundary-root-levelwise-isomorphism}
\end{align}
for every \(a\geq0\), \(\alpha\in\Q\cap[0,1)\), and
\(\lambda\in\mathbb Q\).  For each fixed \(\alpha\), they
also identify the coherent integer-filtered models, their
integer-slice quotients, and their spectral sequences from \(E_1\),
as in \cref{prop:acyclic-defect-descent-criterion}.

If \(\mathscr X\) is proper, then \(\mathscr Y\) is proper and
compactifies the same open LG model.  Morphicity of \(w\) and
projectivity of the coarse moduli space are likewise preserved.  If
\(\mathscr D\) has a finite labeled SNC presentation and
\(\mathscr H\) is one of its labeled members, the rooted boundary
inherits such a presentation.  In particular, a boundary root of a
good stack compactification is again good, and a boundary root of a
projective good stack compactification is again projective good.
\end{theorem}

\begin{proof}
By \cref{prop:boundary-root-basic-properties},
the stack \(\mathscr Y\) is smooth, the divisor
\(\mathscr D_{\mathscr Y}\) is reduced NC, the morphism \(\pi\) is
proper and an isomorphism away from \(\mathscr H\), and
\eqref{eq:boundary-root-structure-sheaf-descent} holds.  Since
\(\mathscr H\subset\mathscr D\), removing the reduced inverse image
of the boundary gives
\(\mathscr U_{\mathscr Y}\simeq\mathscr U_{\mathscr X}\), with
\(\pi\) equal to the identity under this identification.  The
transformed data form an NC rational stack pair by
\cref{lem:boundary-root-pullback-rounding}.  Thus \(\pi\) is an NC
rational stack-pair modification.  Kontsevich pullback is provided
by \cref{prop:stack-kontsevich-functoriality}.

Fix \(\alpha\in\Q\cap[0,1)\), and use the divisor
\(\Delta_\alpha\) of
\cref{lem:boundary-root-pullback-rounding}.  Logarithmic
differentials pull back isomorphically
\cite[Lemma~3.5]{BorneLaaroussi}; locally this is
\(dx/x\mapsto r\,dt/t\).  Since
\(\mathscr P_{\mathscr Y}=\pi^*\mathscr P\) and
\(dw_{\mathscr Y}=\pi^*dw\), the zero-level formula in
\cref{prop:local-kontsevich-yu-package}\textup{(i)} gives
\begin{equation*}
 \pi^*\Omega_{\mathscr X,w}^a(0)
 \simeq\Omega_{\mathscr Y,w_{\mathscr Y}}^a(0).
\end{equation*}
The fractional-twist formula and
\eqref{eq:boundary-root-rounded-divisor} then give
\begin{equation*}
 \Omega_{\mathscr Y,w_{\mathscr Y}}^a(\alpha)
 \simeq
 \pi^*\Omega_{\mathscr X,w}^a(\alpha)
 \otimes\mathcal O_{\mathscr Y}(\Delta_\alpha).
\end{equation*}

Let \(Q_\alpha^a\) be the cokernel of Kontsevich pullback.  Under the
last identification, pullback is the natural inclusion into the
effective twist.  Filter this cokernel by adding the components of
\(\Delta_\alpha\) one at a time.  On the root chart,
the quotient arising from the \(j\)-th twist along
\(\mathscr H_{\mathscr Y,\nu}\) is generated by \(t^{-j}\); its
coefficient is pulled back from \(\mathscr X\), hence has trivial
relative stabilizer action, while \(t^{-j}\) has character
\(\chi^{-j}\).  This character is nontrivial because
\(1\leq j\leq\delta_{\alpha,\nu}\leq r-1\).  Exact
\(\mu_r\)-invariants therefore kill every quotient, and
\cref{prop:boundary-root-basic-properties}\textup{(ii)} plus
d\'evissage gives
\[
 R\pi_*Q_\alpha^a=0.
\]
The acyclic-defect criterion
\cref{prop:acyclic-defect-descent-criterion} now yields
\eqref{eq:boundary-root-kontsevich-termwise-descent},
\eqref{eq:boundary-root-levelwise-isomorphism}, and all their stated
filtered consequences.

Finally, properness is preserved by composition, morphicity by
composition with \(\pi\), and a labeled SNC boundary by
\cref{prop:boundary-root-basic-properties}\textup{(iii)}.  The relative coarse-space
description in
\cref{prop:boundary-root-basic-properties}\textup{(ii)} shows that the
coarse moduli space is unchanged, so coarse projectivity is preserved
as well.  This agrees with the root transform and relative coarse
pair of \cite[\S3.5, p.~10]{BerghRydh}, and proves the last paragraph.
\end{proof}

\subsection{Boundary-admissible ordinary blowups and their filtered comparison}
\label{subsec:ordinary-blowups-filtered-comparison}

For the main theorem, one could avoid a separate stack proof of the
blowup comparison.  The required blowups are covered by two existing
arguments.  Yu treats the boundary blowups used to compare good
compactifications \cite[Lemma~1.6]{Yu}.  Chen--Yu treat the
zero--pole blowups used to resolve nondegenerate rational
compactifications \cite[\S2.3]{ChenYu}.  These results are stated for
smooth varieties in a global setting that includes properness
hypotheses.  Their blowup calculations, however, are étale-local on
the base and do not use properness.  One may therefore apply them on
étale scheme charts and descend the comparisons to the stack.

We give the details because one slightly more general statement
covers all the blowups used here.  A boundary-admissible center may
contain both polar directions and extra transverse equations.  Thus
Yu's noncancelling blowups, the cancelling zero--pole blowups of
Chen--Yu, and the blowups used for boundary strictification are all
special cases of the theorem.  It also covers more general
boundary-admissible blowups.  The main idea of the proof is inspired
by the two arguments above.

\begin{lemma}\label{lem:smooth-stack-blowup-geometry}
Let \(\mathscr X\) be a smooth Deligne--Mumford stack, let
\(C\hookrightarrow\mathscr X\) be a smooth regular closed immersion
of positive locally constant codimension, and put
\[
 \pi\colon\mathscr X':=\Bl_C(\mathscr X)\longrightarrow\mathscr X.
\]
Let \(E\) be the exceptional divisor.  Then the following hold.

\begin{enumerate}[label=\textup{(\roman*)}]
\item The morphism \(\pi\) is representable and projective,
  \(\mathscr X'\) is smooth, and \(\pi\) is an isomorphism over
  \(\mathscr X\setminus C\).  Moreover,
  \[
   E=\mathbb P(N_{C/\mathscr X})\xrightarrow{\,p\,}C,
   \qquad
   \mathcal O_E(E)=\mathcal O_{E/C}(-1).
  \]
  Adjunction is a
  quasi-isomorphism
  \begin{equation}
   \mathcal O_{\mathscr X}\xrightarrow{\ \sim\ }
   R\pi_*\mathcal O_{\mathscr X'}.
   \label{eq:smooth-stack-blowup-structure-sheaf-descent}
  \end{equation}
  If \(\mathscr X\) is proper, then \(\mathscr X'\) is proper.

\item Let \(\mathscr D\subset\mathscr X\) be an NC divisor having
  normal crossings with \(C\).  Thus, at every geometric point of
  \(C\), there is an étale scheme chart with regular parameters such
  that
  \begin{equation}
   \mathscr D=(x_1\cdots x_r=0),
   \qquad
   C=(x_1,\ldots,x_a,x_{r+1},\ldots,x_b),
   \qquad
   0\leq a\leq r\leq b,
   \label{eq:smooth-stack-blowup-nc-coordinates}
  \end{equation}
  where \(a+b-r>0\).  Put
  \(\mathscr D':=(\pi^{-1}\mathscr D)_{\mathrm{red}}\).
  \begin{enumerate}[label=\textup{(\alph*)}]
  \item The divisor \(\mathscr D'\) is NC.  If \(\mathscr D\) has
    a finite labeled SNC presentation, the nonempty strict
    transforms of its labeled members, together with \(E\), have a
    finite labeled SNC presentation.

  \item Let \(A\) be a smooth effective Cartier divisor having
    simultaneous normal-crossing coordinates with \(C\) and
    \(\mathscr D\).  If \(A'\) is its strict transform and
    \(C=\coprod_\nu C_\nu\), with exceptional divisor \(E_\nu\) over
    \(C_\nu\), then
    \begin{equation}
     \pi^*A
     =A'+\sum_\nu\epsilon_{C_\nu}(A)E_\nu,
     \qquad
     \epsilon_{C_\nu}(A)
     =\begin{cases}
       1,&C_\nu\subset A,\\
       0,&C_\nu\not\subset A.
      \end{cases}
     \label{eq:smooth-divisor-total-transform}
    \end{equation}

  \item Suppose that \(C\) is connected and is a connected component
    of the transverse intersection of two smooth effective Cartier
    divisors \(A\) and \(B\).  Then their strict transforms are
    disjoint above \(C\), and
    \begin{equation*}
     A'\cap E\simeq C\simeq B'\cap E.
    \end{equation*}
    If \(K\) is a further smooth Cartier divisor in simultaneous
    normal-crossing position and \(C\not\subset K\), then
    \begin{equation*}
     A'\cap K'\xrightarrow{\ \sim\ }A\cap K.
    \end{equation*}

  \item Assume in addition that \(C\subset|\mathscr D|\).  Define
    the transverse logarithmic conormal bundle by
    \begin{equation}
     \mathcal T_{C/(\mathscr X,\mathscr D)}^\vee
     :=\operatorname{im}\left(
      \mathcal I_C/\mathcal I_C^2
      \xrightarrow{\,d\,}
      \Omega_{\mathscr X}^1(\log\mathscr D)|_C
     \right).
     \label{eq:intrinsic-logarithmic-conormal-bundle}
    \end{equation}
    It is locally free; write
    \(\mathcal T_{C/(\mathscr X,\mathscr D)}\) for its dual.  If
    \(i\colon E\hookrightarrow\mathscr X'\) is the inclusion, then
    there is a canonical exact sequence
    \begin{equation}
     0\longrightarrow
     \pi^*\Omega_{\mathscr X}^1(\log\mathscr D)
     \longrightarrow
     \Omega_{\mathscr X'}^1(\log\mathscr D')
     \longrightarrow
     i_*\left(
      \mathcal O_{E/C}(-1)\otimes
      p^*\mathcal T_{C/(\mathscr X,\mathscr D)}^\vee
     \right)
     \longrightarrow0.
     \label{eq:intrinsic-logarithmic-blowup-modification}
    \end{equation}
  \end{enumerate}

\item Let \(w\) be a nonzero rational function whose divisor has
  normal crossings with \(C\) near \(C\).  For a connected
  component \(C_\nu\), let \(\nu_{C_\nu}(w)\) be the signed sum of
  the multiplicities of the local branches of \(\divisor(w)\)
  containing \(C_\nu\).  Then
  \begin{equation}
   \divisor(w\circ\pi)
   =\widetilde{\divisor(w)}
    +\sum_\nu\nu_{C_\nu}(w)E_\nu.
   \label{eq:rational-function-divisor-transform}
  \end{equation}
  Near \(E\), the support of this divisor is NC.
\end{enumerate}
All the constructions and exact sequences above commute with étale
base change.
\end{lemma}

\begin{proof}
Blowup is the relative Proj of the Rees algebra, so it is
representable, projective, and compatible with flat base change.  On
an étale chart on which the center is a coordinate subspace, the
standard blowup charts prove smoothness and the description of the
exceptional divisor.  They also show that the reduced total transform
of a coordinate NC divisor is again coordinate NC.  Formula
\eqref{eq:smooth-divisor-total-transform} is the order of a local
equation along the ideal of the center.

The structure-sheaf assertion is likewise \'etale-local on
\(\mathscr X\).  On a scheme chart it is the standard relative-Proj
calculation for the blowup of an ideal generated by a regular
sequence: the augmented \v Cech complex of the standard affine
blowup charts has degree-zero cohomology \(\mathcal O_{\mathscr X}\)
and no higher cohomology.  Proper flat base change for the
representable morphism \(\pi\), followed by faithfully flat descent
from an \'etale atlas, proves
\eqref{eq:smooth-stack-blowup-structure-sheaf-descent}.

For \textup{(ii)(c)}, use coordinates
\(A=(u=0)\), \(B=(v=0)\), and \(C=(u=v=0)\).
The charts \(v=ut\) and \(u=vs\) show that \(A'\) and \(B'\) do
not meet above \(C\), while their intersections with \(E\) are the
two sections of \(E\to C\).  If \(K\) is present, the induced
center in \(A\cap K\) is Cartier, so its blowup is an isomorphism.

For \textup{(ii)(d)}, the claim is trivial when the center is a
Cartier divisor, since its blowup is the identity.  Assume that its
codimension is at least two.  In the coordinates
\eqref{eq:smooth-stack-blowup-nc-coordinates}, the image in
\eqref{eq:intrinsic-logarithmic-conormal-bundle} is generated by
\(dx_{r+1},\ldots,dx_b\).  It is therefore locally free.  The blowup
charts give the exact tangent sequence
\[
 0\longrightarrow
 \mathcal T_{\mathscr X'}(-\log\mathscr D')
 \longrightarrow
 \pi^*\mathcal T_{\mathscr X}(-\log\mathscr D)
 \longrightarrow
 i_*p^*\mathcal T_{C/(\mathscr X,\mathscr D)}
 \longrightarrow0.
\]
Only the transverse generators
\(\partial_{x_{r+1}},\ldots,\partial_{x_b}\) occur in the quotient.
Now dualize.  The quotient is supported on the Cartier divisor \(E\),
so its \(\mathcal Hom\) into \(\mathcal O_{\mathscr X'}\) is zero and
\[
 \mathcal Ext^1_{\mathscr X'}(i_*\mathcal F,
 \mathcal O_{\mathscr X'})
 \simeq
 i_*(\mathcal F^\vee\otimes\mathcal O_E(E)).
\]
There are no higher terms.  This gives
\eqref{eq:intrinsic-logarithmic-blowup-modification}.  It is the
intrinsic degree-one form of
\cite[Lemma~1.6(i)]{Yu}.  Finally, the coefficient of \(E_\nu\) in
the pullback of a principal divisor is the order of its local
monomial along \(C_\nu\), which proves
\eqref{eq:rational-function-divisor-transform}.  All arguments are
étale-local and hence descend to \(\mathscr X\).
\end{proof}

\begin{definition}[Boundary-admissible centers]
\label{def:boundary-admissible-center}
Let
\((\mathscr X,\mathscr D,\mathscr Z,\mathscr P,w)\)
be an NC rational stack pair.  A \emph{boundary-admissible center}
is a smooth closed substack \(C\subset\mathscr X\) satisfying the
following three conditions.
\begin{enumerate}[label=\textup{(\roman*)}]
\item The center is contained in the boundary,
  \(C\subset|\mathscr D|\), and has normal crossings with
  \(\mathscr D\).  Equivalently, étale-locally at every geometric
  point of \(C\), there are regular parameters and integers
  \(1\leq a\leq r\leq b\) such that
\begin{equation}
 \mathscr D=(x_1\cdots x_r=0),
 \qquad
 C=(x_1,\ldots,x_a,x_{r+1},\ldots,x_b).
 \label{eq:good-model-center-coordinates}
\end{equation}

\item Near \(C\cap|\mathscr P|\), the center has normal crossings
  with the reduced NC divisor \(\mathscr D+\mathscr Z\).

\item No connected component of \(C\) contained in the zero divisor
  is allowed to meet the polar divisor properly: for every connected
  component \(C_\nu\),
  \begin{equation}
   C_\nu\subset|\mathscr Z|,
   \quad C_\nu\cap|\mathscr P|\ne\varnothing
   \quad\Longrightarrow\quad
   C_\nu\subset|\mathscr P|.
   \label{eq:rationally-admissible-cancellation-condition}
  \end{equation}
\end{enumerate}
Thus a component contained in \(\mathscr Z\) either misses the polar
locus or is contained in
\(|\mathscr Z|\cap|\mathscr P|\); it cannot cross the polar locus
transversely.

At every geometric point of \(C\cap|\mathscr P|\),
conditions~\textup{(ii)}--\textup{(iii)} amount to requiring the
following simultaneous coordinate form.  After passing to an étale
chart, there are regular parameters
\[
 x_1,\ldots,x_r,\qquad y_1,\ldots,y_m,
\]
subsets
\(A\subset\{1,\ldots,r\}\) and
\(B\subset\{1,\ldots,m\}\), with \(A\ne\varnothing\), and a unit
\(u\) such that
\begin{align}
 \mathscr D&=(x_1\cdots x_r=0),
 &C&=(x_i\ (i\in A),\ y_j\ (j\in B)),
 \label{eq:rationally-admissible-boundary-coordinates}\\
 \mathscr P&=\sum_{i=1}^q e_i(x_i=0),
 &w&=\frac{u}{x_1^{e_1}\cdots x_q^{e_q}}
 \quad\text{or}\quad
 \frac{u y_1}{x_1^{e_1}\cdots x_q^{e_q}},
 \notag
\end{align}
where every \(e_i>0\).  In the second form,
\eqref{eq:rationally-admissible-cancellation-condition} says exactly
that
\[
 1\in B
 \quad\Longrightarrow\quad
 A\cap\{1,\ldots,q\}\ne\varnothing.
\]

For a connected component \(C_\nu\subset C\), define its
\emph{cancellation index} by
\begin{equation*}
 \varepsilon_\nu
 :=\begin{cases}
  1,&C_\nu\subset|\mathscr Z|
       \text{ and }C_\nu\subset|\mathscr P|,\\
  0,&\text{otherwise}.
 \end{cases}
\end{equation*}
The blowup of \(C\), equipped with
\(\mathscr D'=(\pi^{-1}\mathscr D)_{\mathrm{red}}\) and
\(w'=w\circ\pi\), is called a \emph{boundary-admissible blowup}.
For a morphic potential, conditions~\textup{(ii)}--\textup{(iii)}
are automatic once condition~\textup{(i)} holds, and all cancellation
indices are zero.
\end{definition}

\begin{remark}[The three cases used below]
Three special cases occur in the proof of the main theorem.  For
toroidal boundary strictification, \(C\) is a boundary stratum, so
\(b=r\) in \eqref{eq:good-model-center-coordinates} and
\(\varepsilon_\nu=0\).  For the zero--pole blowups of Chen--Yu
\cite[\S2.3]{ChenYu}, one has, after relabeling,
\(C=(x_1,y_1)\) in
\eqref{eq:rationally-admissible-boundary-coordinates} and
\(\varepsilon_\nu=1\).  For the blowups used in weak factorization
of good compactifications \cite[Lemma~1.6]{Yu}, the center may have
the full form \eqref{eq:good-model-center-coordinates}, including
transverse equations, but the potential is morphic and hence
\(\varepsilon_\nu=0\).  The rest of this subsection treats all three
cases by one comparison argument.
\end{remark}

\begin{lemma}\label{lem:rationally-admissible-blowup-preserves-nc-pair}
Let \(C\) be a boundary-admissible center and let
\[
 \pi\colon\mathscr X':=\Bl_C(\mathscr X)\longrightarrow\mathscr X.
\]
Write \(C=\coprod_\nu C_\nu\), let \(E_\nu\) be the exceptional
divisor over \(C_\nu\), and let
\[
 N_\nu:=\operatorname{coeff}_{E_\nu}(\pi^*\mathscr P),
 \qquad
 \delta_\nu:=\operatorname{coeff}_{E_\nu}(\pi^*\mathscr Z).
\]
Then \(\delta_\nu\geq0\), and
\(\delta_\nu\in\{0,1\}\) whenever \(N_\nu>0\).  In all cases,
\(\varepsilon_\nu=\min\{N_\nu,\delta_\nu\}\), and the zero and polar
divisors of \(w'=w\circ\pi\) are
\begin{align}
 \mathscr Z'
 &=\widetilde{\mathscr Z}
   +\sum_\nu(\delta_\nu-\varepsilon_\nu)E_\nu,
 \label{eq:rationally-admissible-zero-transform}\\
 \mathscr P'
 &=\widetilde{\mathscr P}
   +\sum_\nu(N_\nu-\varepsilon_\nu)E_\nu
  =\pi^*\mathscr P-\sum_\nu\varepsilon_\nu E_\nu.
 \label{eq:rationally-admissible-polar-transform}
\end{align}
The transformed data
\((\mathscr X',\mathscr D',\mathscr Z',\mathscr P',w')\)
form an NC rational stack pair, and \(\pi\) is an isomorphism on the
complement of the boundary.  In particular, an NC rational stack
compactification is transformed into another compactification of the
same open LG model.
\end{lemma}

\begin{proof}
The claims are étale-local.  In the coordinates of
\cref{def:boundary-admissible-center}, the center is a coordinate
subspace of the reduced NC divisor \(\mathscr D+\mathscr Z\) near the
polar locus.  Its blowup is smooth and the reduced total transform is
again NC.  If \(N_\nu>0\), then \(C_\nu\subset|\mathscr P|\), so
polar nondegeneracy gives \(\delta_\nu\in\{0,1\}\).  If
\(N_\nu=0\), no exceptional polar component is created and
\(\delta_\nu\) may be larger away from the polar locus.  In either
case, cancelling the common exceptional order
\(\min\{N_\nu,\delta_\nu\}=\varepsilon_\nu\) gives
\eqref{eq:rationally-admissible-zero-transform} and
\eqref{eq:rationally-admissible-polar-transform}.

Condition
\eqref{eq:rationally-admissible-cancellation-condition} is precisely
what is needed when \(\delta_\nu>0\): if the center also meets the
polar locus, it is contained in \(|\mathscr P|\), whence
\(N_\nu>0\), \(\delta_\nu=1\), and the exceptional zero factor
cancels.  Hence near
\(|\mathscr P'|\), the transformed zero divisor is the strict
transform of \(\mathscr Z\); it is smooth and transverse to
\(\mathscr D'\).  This proves polar nondegeneracy.  Since the center
lies in \(|\mathscr D|\), the blowup is an isomorphism on the open
complement.  Properness is preserved because \(\pi\) is projective.
\end{proof}

\begin{lemma}\label{lem:kontsevich-blowup-exceptional-rounding}
In the situation above, fix a connected component \(C_\nu\).  Let
\(e_{\nu,1},\ldots,e_{\nu,k_\nu}\) be the coefficients of the local
polar branches containing \(C_\nu\); this étale-local multiset is
well defined up to permutation.  Put
\(N_\nu=\sum_j e_{\nu,j}\).  For
\(\alpha\in\Q\cap[0,1)\), define
\begin{equation*}
 \Delta_\nu(\alpha)
 :=\operatorname{coeff}_{E_\nu}\left(
  (\lfloor\alpha\mathscr P'\rfloor-\mathscr P')
  -\pi^*(\lfloor\alpha\mathscr P\rfloor-\mathscr P)
 \right).
\end{equation*}
Then
\begin{align}
 \Delta_\nu(\alpha)
 &=\lfloor\alpha(N_\nu-\varepsilon_\nu)\rfloor
   -\sum_{j=1}^{k_\nu}\lfloor\alpha e_{\nu,j}\rfloor
   +\varepsilon_\nu
 \notag\\
 &=\left\lfloor
   \sum_{j=1}^{k_\nu}\{\alpha e_{\nu,j}\}
   +(1-\alpha)\varepsilon_\nu
  \right\rfloor.
 \label{eq:unified-rounding-saturation-defect}
\end{align}
If \(c_\nu=\operatorname{codim}_{\mathscr X}C_\nu\), then
\begin{equation}
 0\leq\Delta_\nu(\alpha)
 \leq\max\{k_\nu+\varepsilon_\nu-1,0\}
 \leq c_\nu-1.
 \label{eq:unified-rounding-saturation-range}
\end{equation}
Equivalently, with
\[
 L_\alpha
 :=\mathcal O_{\mathscr X'}\bigl(
   \pi^*(\lfloor\alpha\mathscr P\rfloor-\mathscr P)
  \bigr),
 \qquad
 L'_\alpha
 :=\mathcal O_{\mathscr X'}
   (\lfloor\alpha\mathscr P'\rfloor-\mathscr P'),
\]
one has
\begin{equation*}
 L'_\alpha
 \simeq
 L_\alpha\left(\sum_\nu\Delta_\nu(\alpha)E_\nu\right).
\end{equation*}
\end{lemma}

\begin{proof}
By
\eqref{eq:rationally-admissible-polar-transform}, the exceptional
coefficient of \(\mathscr P'\) is
\(N_\nu-\varepsilon_\nu\), while
\[
 \operatorname{coeff}_{E_\nu}
 \pi^*\lfloor\alpha\mathscr P\rfloor
 =\sum_{j=1}^{k_\nu}\lfloor\alpha e_{\nu,j}\rfloor.
\]
This gives the first line of
\eqref{eq:unified-rounding-saturation-defect}.  Writing each
\(\alpha e_{\nu,j}\) as its integral and fractional parts gives the
second.  If \(k_\nu+\varepsilon_\nu=0\), its argument is zero and
\(\Delta_\nu=0\).  Otherwise its argument lies in
\([0,k_\nu+\varepsilon_\nu)\), so
\(\Delta_\nu\leq k_\nu+\varepsilon_\nu-1\).
The \(k_\nu\) polar branches containing the center give distinct
boundary equations in its conormal space.  If
\(\varepsilon_\nu=1\), the numerator contributes one further
transverse equation.  Hence
\(k_\nu+\varepsilon_\nu\leq c_\nu\).  Since \(c_\nu\geq1\), this
proves
\eqref{eq:unified-rounding-saturation-range} and the line-bundle
identity.
\end{proof}

We next isolate the two modifications that enter the Kontsevich
lattice: the exceptional line twist and the change in logarithmic
forms after the distinguished \(dw\)-line has been removed.

\begin{lemma}\label{lem:exterior-powers-elementary-modification}
Let
\((\mathscr X,\mathscr D,\mathscr Z,\mathscr P,w)\)
be an NC rational stack pair, let \(C\subset\mathscr X\) be a
boundary-admissible center, and write
\[
 \pi\colon\mathscr X'=\Bl_C(\mathscr X)\longrightarrow\mathscr X,
 \qquad
 i\colon E\hookrightarrow\mathscr X',
 \qquad
 p\colon E=\mathbb P(N_{C/\mathscr X})\longrightarrow C.
\]
Then the following hold.
\begin{enumerate}[label=\textup{(\roman*)}]
\item For a vector bundle \(\mathcal V\) on \(\mathscr X\) and an
  integer \(\Delta\geq0\), the quotient
  \[
   \frac{\pi^*\mathcal V(\Delta E)}{\pi^*\mathcal V}
  \]
  has a filtration with factors
  \begin{equation*}
   i_*\bigl(
    \mathcal O_{E/C}(-t)\otimes p^*(\mathcal V|_C)
   \bigr),
   \qquad 1\leq t\leq\Delta.
  \end{equation*}

\item After replacing \(\mathscr X\) by a sufficiently small open
  neighbourhood of \(|\mathscr P|\), the morphism
  \[
   \mathcal O_{\mathscr X}(-\mathscr P)
   \xrightarrow{\,dw\,}
   \Omega_{\mathscr X}^1(\log\mathscr D)
  \]
  is a locally split injection.  Denote its image by
  \(\Theta_w\), and put
  \[
   M_{\mathscr X}
   :=\Omega_{\mathscr X}^1(\log\mathscr D)/\Theta_w,
   \qquad
   \mathcal T^\vee
   :=\mathcal T_{C/(\mathscr X,\mathscr D)}^\vee.
  \]
  Here and below in \textup{(ii)}, all data, including
  \(C,E,i,p\), are restricted to this neighbourhood and its inverse
  image, and we retain the same notation.
  On a connected component \(C_\nu\), define
  \begin{equation*}
   \mathcal S_w^\vee|_{C_\nu}
   :=
   \begin{cases}
    0,&\varepsilon_\nu=0,\\
    \Theta_w|_{C_\nu},&\varepsilon_\nu=1.
   \end{cases}
  \end{equation*}
  Then \(\mathcal S_w^\vee\) is a subbundle of
  \(\mathcal T^\vee\), of rank \(\varepsilon_\nu\) on
  \(C_\nu\).  Put
  \[
   \mathcal T_{\mathrm{res}}^\vee
   :=\mathcal T^\vee/\mathcal S_w^\vee,
   \qquad
   \Theta:=\pi^*\Theta_w
   =dw'\,\mathcal O_{\mathscr X'}(-\pi^*\mathscr P),
   \qquad
   \Theta':=dw'\,\mathcal O_{\mathscr X'}(-\mathscr P'),
  \]
  and define the vector bundles
  \begin{equation*}
   M:=\pi^*M_{\mathscr X}
     =\frac{\pi^*\Omega_{\mathscr X}^1(\log\mathscr D)}{\Theta},
   \qquad
   M':=\frac{\Omega_{\mathscr X'}^1(\log\mathscr D')}{\Theta'}.
  \end{equation*}
  They fit into an exact sequence
  \begin{equation}
   0\longrightarrow M\longrightarrow M'
   \longrightarrow
   i_*\bigl(
    \mathcal O_{E/C}(-1)\otimes
    p^*\mathcal T_{\mathrm{res}}^\vee
   \bigr)
   \longrightarrow0.
   \label{eq:residual-logarithmic-elementary-modification}
  \end{equation}

  More explicitly, let
  \begin{equation}
   \mathcal H_C
   :=\operatorname{coker}\left(
    \mathcal T_{\mathrm{res}}^\vee
    \lhook\joinrel\longrightarrow M_{\mathscr X}|_C
   \right).
   \label{eq:residual-logarithmic-restriction-image}
  \end{equation}
  The composite
  \[
   \mathcal T^\vee\longrightarrow
   \Omega_{\mathscr X}^1(\log\mathscr D)|_C
   \longrightarrow M_{\mathscr X}|_C
  \]
  has kernel \(\mathcal S_w^\vee\), and hence induces the injection used
  in \eqref{eq:residual-logarithmic-restriction-image}.  The cokernel
  \(\mathcal H_C\) is locally free.
  For every \(m\geq0\), the quotient
  \(\bigwedge^mM'/\bigwedge^mM\) has a finite filtration.  Over the
  exceptional divisor above \(C_\nu\), its nonzero factors are
  \begin{equation}
   i_*\left(
    \mathcal O_{E/C}(-j)\otimes
    p^*\left(
     \bigwedge^{m-k}\mathcal H_C\otimes
     \bigwedge^k\mathcal T_{\mathrm{res}}^\vee
    \right)
   \right),
   \quad
   1\leq j\leq k\leq
   \min\{m,s_\nu-\varepsilon_\nu\},
   \label{eq:residual-exterior-modification-factors}
  \end{equation}
  where
  \(s_\nu=\operatorname{rank}(\mathcal T^\vee|_{C_\nu})\).
  In particular, every factor has the simpler form
  \begin{equation*}
   i_*\bigl(
    \mathcal O_{E/C}(-j)\otimes p^*\mathcal W
   \bigr),
   \qquad 1\leq j\leq s_\nu-\varepsilon_\nu,
  \end{equation*}
  for a vector bundle \(\mathcal W\) on \(C_\nu\).
\end{enumerate}
All constructions in \textup{(ii)} are made on the chosen polar
neighbourhood and its inverse image in \(\mathscr X'\).
\end{lemma}

\begin{proof}
For \textup{(i)}, insert the line bundles
\(\mathcal O_{\mathscr X'}(tE)\),
\(0\leq t\leq\Delta\), and use
\[
 \mathcal O_{\mathscr X'}(tE)/
 \mathcal O_{\mathscr X'}((t-1)E)
 \simeq i_*\mathcal O_{E/C}(-t).
\]

We prove \textup{(ii)} intrinsically; coordinates are needed only to
check two local assertions.  Near a point of \(|\mathscr P|\), write
\[
 w=\frac{u}{x_1^{e_1}\cdots x_q^{e_q}}
 \quad\text{or}\quad
 w=\frac{u z}{x_1^{e_1}\cdots x_q^{e_q}},
\]
where \(u\) is a unit.  Multiplying \(dw\) by the polar monomial
produces a logarithmic form with a unit coefficient.  Indeed, in the
first case a \(d\log x_i\)-coefficient restricts to \(-u e_i\) along
a polar branch.  In the second case, the \(dz\)-coefficient is a unit
where \(z=0\), while away from \((z=0)\) a
\(d\log x_i\)-coefficient is a unit along a polar branch.  Thus,
after shrinking around \(|\mathscr P|\), \(\Theta_w\) is a line
subbundle and \(M_{\mathscr X}\) is a vector bundle.

If \(\varepsilon_\nu=1\), boundary admissibility gives
\(C_\nu\subset I(w)\), and
\cref{lem:indeterminacy-logarithmic-conormal-line} identifies
\(\Theta_w|_{C_\nu}\) with the numerator conormal direction inside
\(\mathcal T^\vee|_{C_\nu}\).  On the other components we take the
zero subbundle.  Since the connected components are open and closed
in \(C\), these pieces form the global subbundle
\(\mathcal S_w^\vee\).  Consequently,
\[
 \operatorname{rank}
 \bigl(\mathcal T_{\mathrm{res}}^\vee|_{C_\nu}\bigr)
 =s_\nu-\varepsilon_\nu.
\]

The polar transformation
\(\pi^*\mathscr P
=\mathscr P'+\sum_\nu\varepsilon_\nu E_\nu\)
and the logarithmic cotangent modification of
\eqref{eq:intrinsic-logarithmic-blowup-modification} give compatible
exact sequences
\begin{align*}
 0&\longrightarrow\Theta\longrightarrow\Theta'
 \longrightarrow
 i_*\bigl(
  \mathcal O_{E/C}(-1)\otimes p^*\mathcal S_w^\vee
 \bigr)
 \longrightarrow0,
 \\
 0&\longrightarrow\pi^*\Omega_{\mathscr X}^1(\log\mathscr D)
 \longrightarrow\Omega_{\mathscr X'}^1(\log\mathscr D')
 \longrightarrow
 i_*\bigl(
  \mathcal O_{E/C}(-1)\otimes p^*\mathcal T^\vee
 \bigr)
 \longrightarrow0.
\end{align*}
On a component with \(\varepsilon_\nu=0\), the first quotient is
zero.  If \(\varepsilon_\nu=1\), it is the quotient of
\(\Theta'(-E_\nu)\subset\Theta'\).  The exceptional symbols are the
inclusion
\(\mathcal S_w^\vee\subset\mathcal T^\vee\), so the snake lemma
gives \eqref{eq:residual-logarithmic-elementary-modification}.

Let
\(\mathcal H=\operatorname{im}(i^*M\to i^*M')\).
Restricting \eqref{eq:residual-logarithmic-elementary-modification}
to \(E\), and using
\[
 \operatorname{Tor}_1^{\mathscr X'}\!\left(
  i_*\bigl(\mathcal O_{E/C}(-1)\otimes
       p^*\mathcal T_{\mathrm{res}}^\vee\bigr),
  \mathcal O_E
 \right)
 \simeq p^*\mathcal T_{\mathrm{res}}^\vee,
\]
gives
\[
 0\longrightarrow p^*\mathcal T_{\mathrm{res}}^\vee
 \longrightarrow i^*M\longrightarrow\mathcal H\longrightarrow0.
\]
Because \(M=\pi^*M_{\mathscr X}\), the first arrow is the pullback of
the induced map
\(\mathcal T_{\mathrm{res}}^\vee\to M_{\mathscr X}|_C\).
Faithful flatness of the projective bundle \(p\) shows that this map is
injective; equivalently, the original map from \(\mathcal T^\vee\) has
kernel \(\mathcal S_w^\vee\).  The same sequence identifies
\(\mathcal H=p^*\mathcal H_C\), and \(\mathcal H_C\) is locally free by
faithfully flat descent.

It remains to take exterior powers.  Filter intrinsically by the
number of factors mapping to the exceptional quotient in
\eqref{eq:residual-logarithmic-elementary-modification}, and then by
powers of the ideal of \(E\).  To compute the associated graded,
choose locally an equation \(t\) of \(E\) and a splitting in which
the modification has the form
\[
 M'=\widetilde{\mathcal H}\oplus\widetilde{\mathcal Q},
 \qquad
 M=\widetilde{\mathcal H}\oplus t\widetilde{\mathcal Q},
 \qquad
 \mathcal Q=
 \mathcal O_{E/C}(-1)\otimes
 p^*\mathcal T_{\mathrm{res}}^\vee.
\]
Filtering by the number \(k\) of
\(\widetilde{\mathcal Q}\)-factors and by their \(t\)-adic order
gives factors
\[
 i_*\left(
  \mathcal O_{E/C}(k-j)\otimes
  \bigwedge^{m-k}\mathcal H\otimes
  \bigwedge^k\mathcal Q
 \right),
 \qquad 1\leq j\leq k.
\]
Substituting
\(\mathcal H=p^*\mathcal H_C\) and the displayed expression for
\(\mathcal Q\) changes the exceptional degree to
\((k-j)-k=-j\), proving
\eqref{eq:residual-exterior-modification-factors}.  The filtration is
intrinsic and commutes with étale pullback, so these local
computations give the stated global filtration.
\end{proof}

\begin{theorem}[Boundary blowup invariance]
\label{thm:adapted-boundary-blowup-levelwise-invariance}
Let
\((\mathscr X,\mathscr D,\mathscr Z,\mathscr P,w)\)
be an NC rational stack pair, let \(C\) be a boundary-admissible
center, and let
\[
 \pi\colon\mathscr X'=\Bl_C(\mathscr X)\longrightarrow\mathscr X
\]
be the associated blowup.  Equip \(\mathscr X'\) with the reduced
inverse-image boundary and the transformed rational data.  Then
\(\pi\) is an NC rational stack-pair modification and, if
\(C=\coprod_\nu C_\nu\),
\begin{equation}
 \mathscr P'
 =\pi^*\mathscr P-\sum_\nu\varepsilon_\nu E_\nu.
 \label{eq:adapted-polar-transform-global}
\end{equation}
For every \(a\geq0\), \(\alpha\in\Q\cap[0,1)\), and
\(\lambda\in\mathbb Q\), pullback induces quasi-isomorphisms
\begin{align*}
 \Omega_{\mathscr X,w}^a(\alpha)
 &\xrightarrow{\ \sim\ }
 R\pi_*\Omega_{\mathscr X',w'}^a(\alpha),
 \\
 \Phi_{\pi,\lambda}\colon
 \Firr^\lambda\mathcal K_{\mathscr X,w}^\bullet
 &\xrightarrow{\ \sim\ }
 R\pi_*\Firr^\lambda\mathcal K_{\mathscr X',w'}^\bullet.
\end{align*}
For each fixed \(\alpha\), the same comparison identifies the
coherent integer-filtered models, their integer-slice quotients, and
their spectral sequences from \(E_1\).  At every rational level, it is
compatible with the maps to total twisted de Rham cohomology; if the
blowup is the identity on the common open model, the two image
subspaces are equal.

In particular, this includes both noncancelling good-model blowups
and the cancelling zero--pole blowups of
\cref{lem:rational-map-indeterminacy-locus}\textup{(ii)}.  By
composition they also apply to the morphic boundary strictifications
of \cref{prop:nc-model-strictification}.
\end{theorem}

\begin{proof}
The transformed data are an NC rational stack pair by
\cref{lem:rationally-admissible-blowup-preserves-nc-pair}, and
\eqref{eq:adapted-polar-transform-global} is
\eqref{eq:rationally-admissible-polar-transform}.  Fix \(a\) and
\(\alpha\), and let
\[
 Q_\alpha^a
 :=\operatorname{coker}\left(
  \pi^*\Omega_{\mathscr X,w}^a(\alpha)
  \longrightarrow
  \Omega_{\mathscr X',w'}^a(\alpha)
 \right).
\]

We first prove that this cokernel is killed by \(R\pi_*\).  The
calculation is étale-local on \(\mathscr X\).  Near a connected
component of the center, write the blowup as
\(\pi_V\colon V'=\Bl_YV\to V\), and choose normal-crossing
coordinates
\[
 D=(x_1\cdots x_r=0),
 \qquad
 Y=(x_1,\ldots,x_\ell,x_{r+1},\ldots,x_{r+s}).
\]
Put
\[
 c:=\ell+s,
 \qquad
 E:=\pi_V^{-1}(Y),
 \qquad
 i\colon E\hookrightarrow V',
 \qquad
 p\colon E\to Y.
\]
Assume first that the chosen center point lies in \(|P|\), and shrink
the chart inside the polar neighbourhood of
\cref{lem:exterior-powers-elementary-modification}\textup{(ii)}.
Let \(k\) be the number of polar branches containing \(Y\), let
\(\varepsilon\) be the cancellation index, and put
\(\Delta=\Delta(\alpha)\).  Define
\[
 L:=\mathcal O_{V'}\bigl(
   \pi_V^*(\lfloor\alpha P\rfloor-P)
  \bigr),
 \qquad
 L':=\mathcal O_{V'}(\lfloor\alpha P'\rfloor-P').
\]
By \cref{lem:kontsevich-blowup-exceptional-rounding},
\begin{equation*}
 L'=L(\Delta E),
 \qquad
 0\leq\Delta\leq\max\{k+\varepsilon-1,0\}.
\end{equation*}

Inside \(\Omega_{V'}^a(*D')\), put
\begin{align*}
 A&:=dw'\wedge\pi_V^*\Omega_V^{a-1}(\log D),
 &B&:=\pi_V^*\Omega_V^a(\log D),\\
 A'&:=dw'\wedge\Omega_{V'}^{a-1}(\log D'),
 &B'&:=\Omega_{V'}^a(\log D').
\end{align*}
With negative exterior degrees interpreted as zero, set
\begin{align*}
 \mathcal K_0&:=L(A+B),
 &\mathcal K_1&:=L'A+LB,\\
 \mathcal K_2&:=L'(A+B),
 &\mathcal K_3&:=L'(A'+B),\\
 \mathcal K_4&:=L'(A'+B').
\end{align*}
The zero-level description
\eqref{eq:zero-level-kontsevich-lattice} identifies the endpoints as
\[
 \mathcal K_0=\pi_V^*\Omega_{V,w}^a(\alpha),
 \qquad
 \mathcal K_4=\Omega_{V',w'}^a(\alpha).
\]

Let \(M\) and \(M'\) be the vector bundles of
\cref{lem:exterior-powers-elementary-modification}\textup{(ii)}.
The intersection identity of
\cref{lemma:intersection-of-logarithmic-and-polar-forms}, applied
before and after the blowup, and a local splitting of the residual
modification give
\[
 A\cap B=A'\cap B
 =dw'\wedge\pi_V^*\Omega_V^{a-1}(\log D)(-\pi_V^*P),
 \qquad
 A'\cap B'
 =dw'\wedge\Omega_{V'}^{a-1}(\log D')(-P').
\]
For the cross identity, use the normal form from
\cref{lem:exterior-powers-elementary-modification}\textup{(ii)}.
If \(t\) is an equation of \(E\), compatible local splittings in
degree one have the form
\[
 \pi_V^*\Omega_V^1(\log D)=\Theta\oplus M,
 \qquad
 \Omega_{V'}^1(\log D')=\Theta'\oplus M',
 \qquad
 \Theta=t^\varepsilon\Theta',
\]
inside the meromorphic cotangent bundle.  Since \(A'\) has the
distinguished \(dw'\)-direction, its intersection with \(B\) is the
first summand
\(\Theta\wedge\bigwedge^{a-1}M=A\cap B\).  Thus
\(A'\cap B=A\cap B\), and in degree \(a\) the same splittings give
\[
 B=(A\cap B)\oplus\bigwedge^aM,
 \qquad
 B'=(A'\cap B')\oplus\bigwedge^aM'.
\]
The displayed intersections and splittings imply
\[
 \mathcal K_0\subset\mathcal K_1\subset\mathcal K_2
 \subset\mathcal K_3\subset\mathcal K_4.
\]
All relevant intersections are split in these local decompositions.
Applying \((X+Y)/X\simeq Y/(X\cap Y)\) at each step, together with
\(A'\cap B=A\cap B\),
\(B/(A\cap B)=\bigwedge^aM\), and the analogous decompositions for
\(A'\) and \(B'\), gives the successive quotients
\begin{align*}
 \mathcal K_1/\mathcal K_0
 &\simeq (L'/L)\otimes
  \bigwedge^{a-1}M\cdot dw',
 \\
 \mathcal K_2/\mathcal K_1
 &\simeq (L'/L)\otimes\bigwedge^aM,
 \\
 \mathcal K_3/\mathcal K_2
 &\simeq L'\otimes
  \left(
   \frac{\bigwedge^{a-1}M'}{\bigwedge^{a-1}M}
  \right)\cdot dw',
 \\
 \mathcal K_4/\mathcal K_3
 &\simeq L'\otimes
  \frac{\bigwedge^aM'}{\bigwedge^aM}.
\end{align*}
Thus pullback is injective.  After the factor \(L\) is included, both
\(\bigwedge^{a-1}M\cdot dw'\) and \(\bigwedge^aM\) are pullbacks of
vector bundles on \(V\).  The first two quotients are therefore controlled
by \cref{lem:exterior-powers-elementary-modification}\textup{(i)}
and have exceptional degrees
\(-1,\ldots,-\Delta\).  For the last two,
\cref{lem:exterior-powers-elementary-modification}\textup{(ii)}
directly gives degrees \(-\Delta-j\), where
\[
 1\leq j\leq s-\varepsilon;
\]
here the term \(-\Delta\) comes from the restriction of
\(L'=L(\Delta E)\) to \(E\).  If a line-twist factor occurs, then
\(\Delta>0\), hence \(k+\varepsilon>0\), and
\[
 1\leq\mu\leq\Delta
 \leq k+\varepsilon-1
 \leq c-1.
\]
For a residual factor, if \(k+\varepsilon>0\), then
\[
 \Delta+j
 \leq(k+\varepsilon-1)+(s-\varepsilon)
 =k+s-1
 \leq\ell+s-1
 =c-1.
\]
If \(k+\varepsilon=0\), then \(\Delta=0\) and
\(\Delta+j=j\leq s\leq c-1\), because \(\ell\geq1\).
It follows that the restriction of \(Q_\alpha^a\) to this chart has
a finite filtration whose nonzero factors are finite direct sums of
\begin{equation}
 i_*\bigl(
  \mathcal O_{E/Y}(-\mu)\otimes p^*\mathcal V
 \bigr),
 \qquad
 1\leq\mu\leq c-1,
 \label{eq:adapted-blowup-exceptional-factor-range}
\end{equation}
with \(\mathcal V\) locally free on \(Y\).

At a center point outside \(|P|\), shrink the chart to be disjoint
from \(P\).  The fractional Kontsevich lattice is then
\(\Omega_V^a(\log D)\).  Applying the intrinsic ideal-adic
exterior-power filtration used in the proof of
\cref{lem:exterior-powers-elementary-modification}\textup{(ii)} to
\eqref{eq:intrinsic-logarithmic-blowup-modification} gives factors
\(\mathcal O_{E/Y}(-j)\otimes p^*\mathcal V\),
\(1\leq j\leq s\).  Since the center contains at least one boundary
direction,
\(s\leq\ell+s-1=c-1\).  Thus
\eqref{eq:adapted-blowup-exceptional-factor-range} holds on every
chart.

For \(1\leq\mu\leq c-1\), the projective-bundle formula gives
\[
 Rp_*\mathcal O_{E/Y}(-\mu)=0.
\]
The projection formula and dévissage therefore kill every factor in
\eqref{eq:adapted-blowup-exceptional-factor-range}, hence
\(R\pi_*Q_\alpha^a=0\).  Proper flat base change for the
representable blowup and faithful flatness of an étale atlas descend
this conclusion from the charts to \(\mathscr X\).

Structure-sheaf descent is part of
\cref{lem:smooth-stack-blowup-geometry}.  The acyclic-defect criterion
of \cref{prop:acyclic-defect-descent-criterion} therefore gives both
displayed quasi-isomorphisms and all fixed-\(\alpha\) consequences.
Total compatibility in
\cref{prop:yu-total-comparison-compatibility} gives the assertion
about image subspaces.

For a morphic potential, the zero and polar fibers are disjoint, so
conditions~\textup{(ii)}--\textup{(iii)} of
\cref{def:boundary-admissible-center} are automatic and every
cancellation index is zero.  The zero--pole centers of
\cref{lem:rational-map-indeterminacy-locus}\textup{(ii)} satisfy
\eqref{eq:rationally-admissible-cancellation-condition} with
\(\varepsilon=1\).  The final assertions follow.
\end{proof}

\begin{remark}[The Chen--Yu defect inside the filtration]
\label{rem:kontsevich-cancelling-versus-noncancelling}
For the minimal zero--pole center \(Y=(z=x_1=0)\), one has
\(c=2\), \(s=\varepsilon=1\), and
\[
 \Delta(\alpha)
 =\lfloor\alpha(e_1-1)\rfloor
  -\lfloor\alpha e_1\rfloor+1
 \in\{0,1\}.
\]
The residual logarithmic bundle has rank zero.  Hence the last two
steps of the five-term filtration vanish.  If \(\Delta(\alpha)=0\),
the first two vanish as well; if \(\Delta(\alpha)=1\), their factors
are copies of \(\mathcal O_{E/Y}(-1)\).  This recovers exactly the
two-chart Kontsevich cokernel in
\cite[\S2.3, formulas~(9)--(12)]{ChenYu}.
\end{remark}
  \section{Global Resolutions and Compactification Independence}
\label{sec:global-resolution-independence}

The local filtered calculations are complete in
\cref{sec:elementary-stacky-modifications}.  This final section has
two global tasks.  First, it replaces an NC rational model by a good
one: toroidal strictification removes boundary-branch monodromy, and
then finitely many globally defined zero--pole blowups resolve the
rational map.  
Second, relative stacky weak
factorization compares the resulting good models.  
We also record the
independent existence theorem for good compactifications and conclude
with the proof of \cref{thm:introduction-nc-rational-compactification-independence}.

\subsection{Toroidal boundary strictification}
\label{subsec:nc-boundary-strictification}

\begin{lemma}\label{lem:coarse-projectivity}
Let
\(f\colon\mathscr Z'\to\mathscr Z\) be a representable projective
morphism of finite-type Deligne--Mumford stacks over \(\C\).  If the
coarse moduli space \(Z\) of \(\mathscr Z\) is a projective scheme,
then the coarse moduli space \(Z'\) of \(\mathscr Z'\) is a projective
scheme.
\end{lemma}

\begin{proof}
Projectivity is \'etale-local on \(Z\).  On a tame quotient chart
over an affine scheme \(U\), write
\[
 \mathscr Z\times_ZU\simeq[U_0/H],
 \qquad
 W:=\mathscr Z'\times_{\mathscr Z}U_0,
\]
where \(U_0\to U\) is finite and \(H\) is finite \'etale.  After a
further \'etale base change, \(H\) is constant.  The space \(W\) is a
projective scheme over \(U_0\), hence over \(U\).  Tame base change
gives
\[
 Z'\times_ZU\simeq W/H
\]
\cite[Corollary~3.3]{AOV}.  The quotient is a quasi-projective scheme by
\cite[Lemma~2.1(2)]{GS}.  The finite surjection \(W\to W/H\) shows that
the quotient is proper over \(U\).  Hence it is
projective over \(U\).  Thus \(Z'\to Z\) is projective.  Since \(Z\)
is projective, \(Z'\) is a scheme and is projective over \(\C\). \end{proof}

Recall that an NC pair \((\mathscr Y,\mathscr E)\) is \'etale-locally
given in regular coordinates by
\[
 E_V=(q_1\cdots q_r=0)
\]
on an \'etale scheme chart \(V\to\mathscr Y\).  After shrinking \(V\) and
separating the connected components of the divisors \((q_i=0)\), these
branches give a finite labeled SNC presentation of \(E_V\).  We call
such a chart \emph{branch-splitting}.  Quasi-compactness provides an
atlas consisting of finitely many such charts, although its \'etale
relation may still permute their labeled branches.

\begin{proposition}\label{prop:nc-model-strictification}
Consider an NC rational stack compactification
\((\mathscr X,\mathscr D,\mathscr Z,\mathscr P,w)\) of
\((\mathscr U,w|_{\mathscr U})\).  There is a finite composite
\begin{equation}
 \sigma\colon
 \mathscr X^{\mathrm{str}}=\mathscr X_N
 \longrightarrow\mathscr X_{N-1}\longrightarrow\cdots
 \longrightarrow\mathscr X_0=\mathscr X.
 \label{eq:strictification-blowup-tower}
\end{equation}
For \(0\leq i\leq N\), equip \(\mathscr X_i\) with the reduced
inverse-image boundary \(\mathscr D_i\), the pullback potential
\(w_i\), and its actual relatively prime zero and polar divisors
\(\mathscr Z_i\) and \(\mathscr P_i\).  Then the following hold.
\begin{enumerate}[label=\textup{(\roman*)}]
\item Every intermediate tuple
  \[
   (\mathscr X_i,\mathscr D_i,\mathscr Z_i,\mathscr P_i,w_i)
  \]
  is an NC rational stack compactification of the same open LG model.
  Every arrow is a blowup along a smooth closed substack of the
  current boundary.  On every branch-splitting \'etale scheme chart, its
  center is a disjoint union of strict transforms of closed coordinate
  intersections.  In particular, the center is boundary-admissible, and
  all its cancellation indices are zero.
\item The final boundary \(\mathscr D_N\) carries a finite labeled SNC
  presentation; equipped with this presentation, the final tuple is
  strict.
\end{enumerate}

Consequently, \(\sigma\) is representable and projective and is an
isomorphism over \(\mathscr U\).  If the original model is morphic, the
final model is good.  If the original model is projective, so is the
final model.
\end{proposition}

We isolate the geometric input used in the proof.  For background, the
classical language of toroidal embeddings and polyhedral subdivisions is
developed in \cite{KKMSD}; logarithmic structures on algebraic stacks are
treated in \cite{Olsson}, and skeletons, fans, and Artin fans of
logarithmic structures in \cite{ACMUW}.

For an NC pair \((\mathscr Y,\mathscr E)\), choose a branch-splitting
atlas \(V\to\mathscr Y\), with \'etale relation
\(R\rightrightarrows V\).  The strata of \(V\) and \(R\), with
their generization and pullback maps, determine a finite diagram of
regular cones and face morphisms.  Its colimit is denoted
\(\Sigma(\mathscr Y,\mathscr E)\); it is independent of the atlas up to
canonical isomorphism
\cite[Section~2.6 and Sections~6.1--6.2]{ACP}.  In the presence of
monodromy, subdivision means compatible subdivision of this diagram
before taking its colimit \cite[Sections~4.2--4.5]{MPS}.  The following
formulation makes explicit the needed factorization into smooth-center
blowups.

\begin{lemma}\label{lem:barycentric-smooth-center-blowups}
Let \((\mathscr Y,\mathscr E)\) be a smooth finite-type
Deligne--Mumford stack with reduced NC boundary, and put
\[
 \Sigma:=\Sigma(\mathscr Y,\mathscr E).
\]
Then the following hold.
\begin{enumerate}[label=\textup{(\roman*)}]
\item The toroidal modification
  \[
   \operatorname{Bar}(\Sigma)\longrightarrow\Sigma
  \]
  is realized by a finite sequence of ordinary blowups along smooth
  centers contained in the current reduced boundary and having normal
  crossings with it; all these morphisms are representable and projective.

  On every branch-splitting \'etale chart, the center at the stage indexed
  by codimension \(k\) is the disjoint union of the strict transforms of
  the closed codimension-\(k\) coordinate intersections; the values of
  \(k\) are treated in decreasing order, with the trivial
  codimension-one blowups omitted.  The construction commutes with
  \'etale base change and descends to \(\mathscr Y\).

\item The generalized cone complex \(\operatorname{Bar}(\Sigma)\) is a
  genuine cone complex: every cone embeds in the colimit, and no
  automorphism identifies two distinct rays in one of its cones.  In
  particular, the final reduced boundary after these blowups carries a
  finite labeled SNC presentation.
\end{enumerate}
\end{lemma}

\begin{proof}
The barycentric subdivision and its realization by successive star
subdivisions are standard; see
\cite[Section~5.3, especially pp.~334--335]{MPS} and
\cite[Lemma~A.2.4]{Harper}.  We record the smooth-center factorization
and descent details needed below.

On a branch-splitting \'etale chart, write
\[
 E=(q_1\cdots q_r=0),\qquad
 L_I:=\bigcap_{i\in I}(q_i=0)
 \quad(\varnothing\ne I\subseteq\{1,\ldots,r\}).
\]
Blowing up \(L_I\) corresponds to the star subdivision of
\(\mathbb R_{\geq0}^{r}\) at
\[
 b_I=\sum_{i\in I}e_i.
\]
Performing these subdivisions in decreasing order of \(|I|\) gives
\(\operatorname{Bar}(\mathbb R_{\geq0}^{r})\).
Immediately before the subsets of cardinality \(k\) are treated, the
strict transforms of the \(L_I\), \(|I|=k\), are pairwise disjoint:
for \(I\ne J\),
\[
 L_I\cap L_J=L_{I\cup J},\qquad |I\cup J|>k,
\]
and this deeper coordinate intersection has already been blown up.  In
the corresponding coordinate blowup, the two strict transforms meet the
exceptional divisor in disjoint coordinate linear subspaces.  The same
coordinate description shows inductively that these strict transforms
are smooth and have normal crossings with the current boundary.
Consequently, every simultaneous subdivision at a fixed codimension is
a blowup along a disjoint union of smooth closed coordinate
intersections.

The construction is symmetric in the local branches.  Inductively,
blowup and strict transform commute with \'etale base change, so the two
pullbacks of the next center ideal to
\(R=V\times_{\mathscr Y}V\) coincide.  Effective \'etale descent produces
the center on the current stack; its smoothness and normal-crossing
position may be checked on \(V\).  This is the groupoid construction of
\cite[proof of Lemma~A.2.6]{Harper}; see also
\cite[Section~5.3, pp.~334--335]{MPS}.  This proves \textup{(i)}.

A cone of \(\operatorname{Bar}(\Sigma)\) is indexed by a strict flag of
faces.  Since its members have distinct dimensions, every automorphism
preserving the flag fixes its barycentric rays and acts trivially on the
corresponding cone.  Thus the subdivided cones embed in the colimit
\cite[Section~5.6, p.~342]{MPS}; compare
\cite[Sections~3.2.1--3.2.3]{HolmesSchwarz}.  Its rays consequently
define global smooth Cartier boundary divisors, and the intersection of
any distinct such divisors is \'etale-locally a disjoint union of
coordinate intersections of the expected codimension.  Splitting these
divisors into their finitely many nonempty connected components gives a
finite labeled SNC presentation, proving \textup{(ii)}.
\end{proof}

\begin{remark}
A second barycentric subdivision would make the cone complex the cone
over a finite simplicial complex, but this is unnecessary here because
the definition of labeled SNC allows intersections to be disconnected
\cite[Section~5.6, p.~342]{MPS}.
\end{remark}

\begin{proof}[Proof of \cref{prop:nc-model-strictification}]
The barycentric modification supplied by
\cref{lem:barycentric-smooth-center-blowups} produces the finite tower
\eqref{eq:strictification-blowup-tower}.  Every arrow is a blowup along
a smooth center whose pullback to a branch-splitting
\'etale chart is a disjoint union of strict transforms of closed
coordinate intersections.  Such a center is contained in the boundary
and has normal crossings with it.
Near the polar locus, the nondegeneracy normal form says that the zero
divisor is either empty or is defined by a parameter transverse to all
boundary coordinates.  Hence the center has normal crossings with
\(\mathscr D_i+\mathscr Z_i\), and no component meeting
\(\mathscr P_i\) is contained in \(\mathscr Z_i\).  It is therefore
boundary-admissible, with all cancellation indices equal to zero.
Applying
\cref{lem:rationally-admissible-blowup-preserves-nc-pair} at each stage
proves that all intermediate tuples are NC rational stack
compactifications of the same open LG model.  This proves
\textup{(i)}.

By \cref{lem:barycentric-smooth-center-blowups}\textup{(ii)}, the
final boundary carries a finite labeled SNC presentation.  Equip
\(\mathscr D_N\) with this presentation.  The final tuple is then
strict, proving \textup{(ii)}.

Finally, these blowups are representable and projective, and their
centers lie in the boundary, so \(\sigma\) is an isomorphism over
\(\mathscr U\).  Morphicity is preserved by pullback; hence a morphic
initial model yields a good final model.  If the initial coarse moduli
space is projective, repeated application of
\cref{lem:coarse-projectivity} proves the same for the final one.
\end{proof}

\subsection{Good resolutions of NC rational models}
\label{subsec:resolution-nc-rational-models}

The indeterminacy locus
\[
 I=(|\mathscr Z|\cap|\mathscr P|)_{\mathrm{red}}
\]
need not be smooth.  At a point where several polar branches meet,
its local form is
\[
 I=(z,x_1\cdots x_\ell)
   =\bigcup_{i=1}^{\ell}(z,x_i).
\]
The whole locus is regularly immersed in codimension two, but it can
be singular.  Hence it is not, in general, a center to which the
smooth blowup comparison of
\cref{thm:adapted-boundary-blowup-levelwise-invariance} applies.  See
\cref{lem:rational-map-indeterminacy-locus}\textup{(i),(iii)}.

We therefore first apply the toroidal strictification of
\cref{prop:nc-model-strictification}.  On the resulting strict model,
each connected labeled polar component \(H\) is a smooth Cartier
divisor.  Each connected component of its intersection with the
transformed zero divisor is then a smooth
codimension-two closed substack with normal crossings with the
boundary; see
\cref{lem:rational-map-indeterminacy-locus}\textup{(ii)}.  We blow up
these smooth zero--pole branches one at a time.  This is the
smooth-variety algorithm of Chen--Yu \cite[\S2.3]{ChenYu}.  Both the
strictification blowups and the zero--pole blowups are
boundary-admissible, so the filtered comparison applies at every
step.

\begin{theorem}[Good resolution of an NC rational model]
\label{thm:filtered-good-resolution}
Let
\((\mathscr X,\mathscr D,\mathscr Z,\mathscr P,w)\)
be an NC rational stack compactification.  There is a factorization
of a representable projective birational morphism
\[
 \mathscr Y\xrightarrow{\ \pi\ }\mathscr X^{\mathrm{str}}
 \xrightarrow{\ \tau\ }\mathscr X,
 \qquad \rho=\tau\circ\pi,
\]
with the following properties.
\begin{enumerate}[label=\textup{(\roman*)}]
\item The morphism \(\tau\) is a toroidal boundary
  strictification, and
  \(\mathscr X^{\mathrm{str}}\) is a strict NC rational stack
  compactification.  If the original model is already strict, one may take
  \(\tau=\id_{\mathscr X}\).
\item Put \(\mathscr X_0=\mathscr X^{\mathrm{str}}\) and
  \(w_0=w\circ\tau\).  The morphism \(\pi\)
  factors as a finite tower
  \begin{equation*}
   \mathscr Y=\mathscr X_N
   \xrightarrow{\ \varpi_{N-1}\ }\mathscr X_{N-1}
   \longrightarrow\cdots\longrightarrow
   \mathscr X_1\xrightarrow{\ \varpi_0\ }
   \mathscr X_0=\mathscr X^{\mathrm{str}}.
  \end{equation*}
  Writing \(w_n=w_0\circ\varpi_0\circ\cdots\circ\varpi_{n-1}\) and
  \(\divisor(w_n)=\mathscr Z_n-\mathscr P_n\), the center at stage
  \(n\) is a connected component
  \(C_n\) of \(\mathscr Z_n\times_{\mathscr X_n}H_n\), for a
  connected polar boundary component \(H_n\).  It is smooth,
  regularly immersed in codimension two, and has normal crossings
  with the labeled boundary.  Every intermediate model is strict NC
  rational, and every
  \(\varpi_n\) is representable, projective, and an isomorphism over
  \(\mathscr U\).  If \(H_n\) has pole order \(e_n\), the exceptional
  divisor has pole order \(e_n-1\), and is nonpolar when \(e_n=1\).
\item Both morphisms are isomorphisms over \(\mathscr U\), and the
  pullback of \(w\) to \(\mathscr Y\) extends to a morphism
  \(w_{\mathscr Y}\colon\mathscr Y\to\mathbb P^1\).  Thus
  \((\mathscr Y,\boldsymbol{\mathscr D}_{\mathscr Y},
  w_{\mathscr Y})\) is a good stack compactification.
\end{enumerate}
If the original model is projective NC rational, then the
strictified model and the final good model are projective.
\end{theorem}

\begin{proof}
Choose \(\tau\) as in \cref{prop:nc-model-strictification}.  Its
source is strict NC rational; if the original model is already
strict, use the empty strictification.
Put \(w_0=w\circ\tau\).  On each successive model write
\[
 \divisor(w_n)=\mathscr Z_n-\mathscr P_n,
 \qquad
 \mathscr P_n=\sum_{H\in\mathscr A_n}e_n(H)H,
\]
let \(\mathcal I_n\) consist of the pairs \((H,C)\) with
\(e_n(H)>0\) and \(C\) a connected component of
\(\mathscr Z_n\times_{\mathscr X_n}H\), and set
\[
 \beta_n:=\sum_{(H,C)\in\mathcal I_n}e_n(H).
\]

If \(\beta_n>0\), choose \((H_n,C_n)\in\mathcal I_n\) and blow up
\(C_n\).  By
\cref{lem:rational-map-indeterminacy-locus}\textup{(ii)}, this is a
smooth global center of codimension two having normal crossings with
the boundary.  The local form in that lemma and
\cref{lem:smooth-stack-blowup-geometry}\textup{(i),(ii),(iii)} show
that the blowup has all the asserted geometric properties, preserves
the NC rational condition, and gives the exceptional divisor pole
order \(e_n(H_n)-1\).  Moreover,
part~\textup{(ii)(c)} of the blowup lemma shows that the selected
zero--pole intersection is transferred to the unique intersection
with the exceptional divisor \(E_n\).  This gives an exceptional
incidence of weight \(e_n(H_n)-1\) when \(e_n(H_n)>1\), and no polar
incidence when \(e_n(H_n)=1\).  We spell out why all other incidences,
including their connected components, are unchanged.  The strict
transforms of \(\mathscr Z_n\) and \(H_n\) are disjoint, whereas
\(\mathscr Z_{n+1}\cap E_n\to C_n\) is an isomorphism.  If
\(K\ne H_n\) is another polar boundary component, the SNC condition
implies that \(C_n\not\subset K\); otherwise the three local
equations for the zero divisor, \(H_n\), and \(K\) would fail to
have the expected codimension.  The final assertion of
\cref{lem:smooth-stack-blowup-geometry}\textup{(ii)(c)} therefore
identifies
\(\mathscr Z_{n+1}\cap K'\) with
\(\mathscr Z_n\cap K\), so its connected components neither split
nor merge.  Thus the selected summand of weight \(e_n(H_n)\) is
replaced by one of weight \(e_n(H_n)-1\) if this is positive, and
otherwise disappears; every other summand of \(\beta_n\) is
preserved.  Hence
\[
 \beta_{n+1}
 =\beta_n-e_n(H_n)+\bigl(e_n(H_n)-1\bigr)
 =\beta_n-1.
\]
The construction therefore stops after exactly \(N=\beta_0\) steps.

On the final model \(\beta_N=0\), equivalently
\(|\mathscr Z_N|\cap|\mathscr P_N|=\varnothing\).  Thus
\cref{lem:rational-map-indeterminacy-locus}\textup{(i)} shows that the
pullback of \(w\) extends to a morphism
\(\mathscr Y\to\mathbb P^1\), so the final model is good.
Representability, projectivity, and preservation of
\(\mathscr U\) follow from the two construction steps.  Under the
projective hypothesis, projectivity of the strictification and
repeated application of \cref{lem:coarse-projectivity} give the final
geometric assertion.
\end{proof}

Thus the only stack-specific global issue is the monodromy of boundary
branches; strictification removes it.  Once the centers are defined on
the stack, both termination and filtered comparison are supplied by
the preceding geometric and elementary-modification results.

\subsection{Existence of good compactifications}
\label{subsec:existence-good-compactifications}

\begin{proposition}\label{prop:good-stack-compactification-existence}
Let \((\mathscr U,w)\) be a smooth separated Deligne--Mumford
Landau--Ginzburg model.
\begin{enumerate}[label=\textup{(\roman*)}]
\item The model \((\mathscr U,w)\) admits a good stack
  compactification.
\item If the coarse moduli space of \(\mathscr U\) is a
  quasi-projective scheme, then the compactification can be chosen
  projective good.
\end{enumerate}
\end{proposition}

\begin{proof}
We construct the compactification in four stages.

\smallskip
\noindent
\emph{Step 1: a proper ambient stack.}
Over \(\C\), every Deligne--Mumford stack is strictly tame.  The
compactification theorem
\cite[Theorem~F]{RydhCompactification} therefore gives an open
immersion
\[
 \iota\colon\mathscr U\hookrightarrow\mathscr W_0
\]
into a proper Deligne--Mumford stack.  Replacing \(\mathscr W_0\)
by the scheme-theoretic closure of \(\mathscr U\) removes any
component disjoint from the open and makes \(\mathscr U\) dense in
every component.

Under the additional hypothesis in \textup{(ii)}, we use instead
\cite[Theorems~4.4 and~5.3]{KreschGeometry}.  The hypotheses of the
two results should be checked in this order.  Since \(\mathscr U\) is
smooth and separated, is generically tame, and has quasi-projective
coarse space, Theorem~4.4 first shows that \(\mathscr U\) is a
quotient stack.  Theorem~5.3 then applies to this quotient stack and
gives a locally closed immersion of \(\mathscr U\) into a smooth
proper Deligne--Mumford stack \(\mathscr W\) with projective coarse
moduli space.  Take \(\mathscr W_0\) to be the scheme-theoretic closure of
\(\mathscr U\) in \(\mathscr W\).  Then \(\mathscr U\) is open and
componentwise dense in \(\mathscr W_0\), and
\cref{lem:coarse-projectivity}, applied to the closed immersion into
\(\mathscr W\), shows that \(\mathscr W_0\) has projective coarse
moduli space.  Thus the remainder of the construction is common to
the two cases, with projectivity retained in case~\textup{(ii)}.

\smallskip
\noindent
\emph{Step 2: make the rational function into a morphism.}
Consider the graph immersion
\[
 (\iota,[w:1])\colon
 \mathscr U\longrightarrow\mathscr W_0\times\mathbb P^1
\]
and let \(\mathscr G\) be its scheme-theoretic closure.  Then
\(\mathscr G\) is a proper Deligne--Mumford stack and the graph
identifies \(\mathscr U\) with a dense open substack of every
component of \(\mathscr G\).  The second projection restricts to a
morphism
\[
 q\colon\mathscr G\longrightarrow\mathbb P^1
\]
whose restriction to \(\mathscr U\) is \([w:1]\).  Thus the graph
closure resolves the indeterminacy at the level of a possibly
singular compactification.

We may and do replace \(\mathscr G\) by its reduction.  This does not
change the open substack \(\mathscr U\), the morphism \(q\), or
properness, and it places the stack directly in the reduced
characteristic-zero category of the resolution theorem used below.

The closed immersion
\(\mathscr G\hookrightarrow\mathscr W_0\times\mathbb P^1\) is
representable and projective.  In case~\textup{(ii)}, its target has
projective coarse space, so \cref{lem:coarse-projectivity} shows that
the coarse space of \(\mathscr G\) is projective.  Notice that no
assertion about smoothness or the boundary has yet been made.

\smallskip
\noindent
\emph{Step 3: functorial stack desingularization with boundary.}
Let
\[
 \mathscr B:=(\mathscr G\setminus\mathscr U)_{\mathrm{red}}.
\]
Temkin's functorial desingularization with boundary
\cite[Theorem~1.1.11]{TemkinEmbedded} gives, for schemes, a finite
\((\mathscr B\cup\mathscr G_{\mathrm{sing}})\)-supported sequence
whose source is regular and on which the total transform of
\(\mathscr B\) is strictly monomial.  This is the weak
principalization/desingularization conclusion needed here; the
cited theorem does not assert that every center in the sequence is
smooth.  More importantly for the
present application, the stack extension is part of the theorem:
\cite[Theorem~1.1.13(i)]{TemkinEmbedded} states that this functor
induces the analogous blowup functor on quasi-compact
quasi-excellent algebraic stacks.  Thus no additional atlaswise
choice of centers is involved.

Apply that stack functor to \((\mathscr G,\mathscr B)\).  Since
\(\mathscr U\) is smooth,
\(\mathscr G_{\mathrm{sing}}\subset\mathscr B\).  We obtain a finite
sequence of representable projective blowups
\[
 \rho_0\colon\mathscr X_0\longrightarrow\mathscr G
\]
supported over \(\mathscr B\), hence an isomorphism over
\(\mathscr U\), such that \(\mathscr X_0\) is smooth and
\[
 \mathscr D_0
 :=(\mathscr X_0\setminus\mathscr U)_{\mathrm{red}}
\]
is NC.  Concretely, strict monomiality on each scheme chart says
that \(\mathscr D_0\) is a union of coordinate hyperplanes there.
We only claim NC at this point: the stack relation may still identify
or permute those local branches, which is why Step~4 remains
necessary.

The composite
\[
 w_0:=q\circ\rho_0
 \colon\mathscr X_0\longrightarrow\mathbb P^1
\]
is a morphism extending \(w\).  Every blowup used above is
representable and projective.  Hence \(\mathscr X_0\) remains
proper.  In case~\textup{(ii)}, repeated application of
\cref{lem:coarse-projectivity}, starting from the projective coarse
space of \(\mathscr G\), shows that the coarse space of
\(\mathscr X_0\) is projective.  We have therefore produced a
morphic NC model, projective under the additional hypothesis.

\smallskip
\noindent
\emph{Step 4: a finite labeled SNC presentation of the boundary.}
The NC divisor \(\mathscr D_0\) can still have self-intersections or
stabilizer monodromy, so the preceding step does not yet give a
good compactification in the sense of
\cref{def:nc-rational-stack-pair}.  Apply
\cref{prop:nc-model-strictification} to
\((\mathscr X_0,\mathscr D_0,w_0)\).  Its centers lie in the
boundary, hence the strictification remains an isomorphism over
\(\mathscr U\); the morphism to \(\mathbb P^1\) is preserved by
composition, and the final boundary is finite labeled SNC.  Since
strictification preserves projectivity by the final assertion of
\cref{prop:nc-model-strictification}, the result is a good stack
compactification in general and a projective good stack
compactification in case~\textup{(ii)}.
\end{proof}

The general existence assertion in part~\textup{(i)} uses Rydh's
compactification theorem in its preliminary version
\cite[Theorem~F]{RydhCompactification} and is logically separate from
compactification independence.  Part~\textup{(ii)} instead starts from
Kresch's published compactification results
\cite[Theorems~4.4 and~5.3]{KreschGeometry}.

 \subsection{Relative factorization of good models}
\label{subsec:stacky-weak-factorization-relative-open}

The second global input is weak factorization relative to the common
open model.  Its intermediate stacks are proper good models, but
their coarse spaces need not be projective; the elementary filtered
comparisons require no such projectivity.

\begin{theorem}[Relative stacky weak factorization]
\label{thm:relative-stacky-weak-factorization}
Let
\((\mathscr X_1,\boldsymbol{\mathscr D}_1,w_1)\) and
\((\mathscr X_2,\boldsymbol{\mathscr D}_2,w_2)\) be two good stack
compactifications of the same smooth LG model \((\mathscr U,w)\).
There is a finite zigzag
\begin{equation*}
 \mathscr X_1=\mathscr Y_0
 \dashrightarrow\mathscr Y_1\dashrightarrow\cdots
 \dashrightarrow\mathscr Y_N=\mathscr X_2
\end{equation*}
such that:
\begin{enumerate}[label=\textup{(\roman*)}]
\item every \(\mathscr Y_i\) is smooth and proper, contains
  \(\mathscr U\) as a componentwise dense open, and its complement
  has a finite labeled SNC presentation;
\item one orientation of every edge is an elementary admissible
  modification, proper and equal to the identity over
  \(\mathscr U\);
\item there is an index \(i_0\) such that the induced birational map
  \(\mathscr Y_i\dashrightarrow\mathscr X_1\) is a morphism for
  \(i\leq i_0\), while
  \(\mathscr Y_i\dashrightarrow\mathscr X_2\) is a morphism for
  \(i\geq i_0\);
\item the potential extends uniquely to compatible morphisms
  \(w_i\colon\mathscr Y_i\to\mathbb P^1\).  Hence every
  \((\mathscr Y_i,\mathscr Y_i\setminus\mathscr U,w_i)\) is a good
  stack compactification.
\end{enumerate}
\end{theorem}

\begin{proof}
Choose an ordering of each finite label set.  Then
\((\mathscr X_1,\boldsymbol{\mathscr D}_1)\) and
\((\mathscr X_2,\boldsymbol{\mathscr D}_2)\) are standard pairs in
the sense of \cite[Definitions~3.1 and~3.5]{BerghRydh}.  The closure
of the diagonal copy of \(\mathscr U\) in
\(\mathscr X_1\times\mathscr X_2\), and hence its normalization, has
proper projections to both factors.  Thus the induced birational map
is proper in the graph sense and is the identity over \(\mathscr U\).
Stacky weak
factorization gives a zigzag by stacky blowups and blowdowns, all
supported in the boundary, together with the SNC and turning-index
properties
\cite[Theorem~D]{BerghRydh}; see also
\cite[Theorem~1.1]{Harper}.  The published global-quotient case is
\cite[Theorem~1.2]{BerghWF}.

By definition, a stacky blowup of weight \(r\) is the ordinary
blowup of its smooth center followed by the \(r\)-th root along the
exceptional divisor \cite[Definition~3.1]{BerghWF}.  Refine the
zigzag by inserting this intermediate ordinary blowup for every
edge; when \(r=1\), there is no nontrivial root step.  The center is
contained in the boundary and has normal crossings with it, so the
strict transforms of the old labeled members, together with the
exceptional divisor, give a finite labeled SNC presentation on each
new model.  Ordinary blowups and roots are proper and are the
identity over \(\mathscr U\).  After reindexing, the turning-index
property is unchanged.  It also shows that every intermediate model
maps properly to one of the proper endpoints and hence is proper.

Use the turning-index morphisms to compose the models on the two
sides with \(w_1\) and \(w_2\), respectively.  On the turning model
the two resulting maps to \(\mathbb P^1\) agree over \(\mathscr U\).
They therefore agree everywhere: the source is reduced,
\(\mathscr U\) is componentwise dense, and \(\mathbb P^1\) is
separated.  The same argument along every edge gives compatible
morphisms
\(w_i\colon\mathscr Y_i\to\mathbb P^1\).

It remains only to identify the refined edges with the operations of
\cref{def:elementary-admissible-modification}.  For an ordinary
blowup edge, the potential on its target is morphic and its center is
a smooth boundary center having normal crossings with the boundary;
hence the remaining conditions in
\cref{def:boundary-admissible-center} are automatic.  The root edge
is taken along its labeled exceptional boundary divisor.  Thus each
refined edge, in one orientation, is an elementary admissible
modification, and every intermediate triple is a good stack
compactification.  The chosen ordering was used only to invoke weak
factorization and is not part of the resulting compactification data.
\end{proof}

\subsection{Proof of the main theorem}
\label{subsec:proof-main-nc-rational-independence}

\begin{theorem}[Compactification independence]
\label{thm:detailed-nc-rational-compactification-comparison}
Let \((\mathscr U,w)\) be a smooth separated Deligne--Mumford
Landau--Ginzburg model of finite type over \(\mathbb C\), and let
\(\mathscr X_1\) and \(\mathscr X_2\) be two NC rational stack
compactifications.  Then, for every \(k\in\mathbb Z\) and
\(\lambda\in\mathbb Q\),
\begin{equation*}
 F_{\mathrm{irr},\mathscr X_1}^{\lambda}H^k(\mathscr U,w)
 =
 F_{\mathrm{irr},\mathscr X_2}^{\lambda}H^k(\mathscr U,w)
\end{equation*}
as subspaces of \(H^k(\mathscr U,w)\).  In particular, their common
value depends only on \((\mathscr U,w)\).
\end{theorem}

\begin{proof}
For \(i=1,2\), choose a good resolution
\(\rho_i\colon\mathscr Y_i\to\mathscr X_i\) as in
\cref{thm:filtered-good-resolution}.  Its strictification and
zero--pole steps are boundary-admissible blowups.  Hence
\cref{thm:adapted-boundary-blowup-levelwise-invariance} and
\cref{cor:levelwise-images-from-total-comparison} identify, at every
rational level, the image filtration defined by \(\mathscr X_i\)
with that defined by \(\mathscr Y_i\).

Join \(\mathscr Y_1\) and \(\mathscr Y_2\) by the zigzag of
\cref{thm:relative-stacky-weak-factorization}.  In one orientation,
each edge is either a boundary root or a boundary-admissible blowup.
Apply \cref{thm:boundary-root-filtered-invariance} or
\cref{thm:adapted-boundary-blowup-levelwise-invariance}, followed by
\cref{cor:levelwise-images-from-total-comparison}, to each edge.
Every edge is the identity over \(\mathscr U\), so all these
equalities take place inside the same vector space
\(H^k(\mathscr U,w)\).  Composing them along the two resolutions and
the zigzag proves the asserted equality.  Since it is an equality of
subspaces, it is independent of all choices.
\end{proof}

\begin{corollary}[Yu spectral-sequence invariance]
\label{cor:compactification-invariance-yu-spectral-sequences}
In the setting of
\cref{thm:detailed-nc-rational-compactification-comparison}, write
\(w_i\) for the rational extension of \(w\) to \(\mathscr X_i\).
Fix \(\alpha\in\Q\cap[0,1)\), and put
\begin{equation*}
 F_{\alpha,i}^p
 :=F_{\mathrm{irr}}^{p-\alpha}
   \mathcal C_{\mathscr X_i,w_i}^\bullet,
 \qquad p\in\mathbb Z.
\end{equation*}
Then there is an isomorphism, from the \(E_1\)-page onward, between
the hypercohomology spectral sequences
\begin{equation*}
 E_1^{p,q}(\mathscr X_i,\alpha)
 =
 \mathbb H^{p+q}\!\left(
  \mathscr X_i,
  F_{\alpha,i}^p/F_{\alpha,i}^{p+1}
 \right)
 \Longrightarrow
 \mathbb H^{p+q}\!\left(
  \mathscr X_i,
  \mathcal C_{\mathscr X_i,w_i}^\bullet
 \right),
 \qquad i=1,2.
\end{equation*}
The isomorphism may depend on the choices of resolutions,
factorization, and inverse edge isomorphisms.  In particular, the
spectral sequence for \(\mathscr X_1\) degenerates at \(E_1\) if and
only if the spectral sequence for \(\mathscr X_2\) does.
\end{corollary}

\begin{proof}
Choose the good resolutions and relative stacky weak factorization
used in the proof of
\cref{thm:detailed-nc-rational-compactification-comparison}.  For
each resolution or factorization edge, the fixed-\(\alpha\) forms of
\cref{thm:boundary-root-filtered-invariance} and
\cref{thm:adapted-boundary-blowup-levelwise-invariance} identify the
coherent integer-filtered models and their successive quotients.
After taking derived global sections, invert the comparison for a
backward edge and compose along the zigzag.  The resulting filtered
comparison induces an isomorphism of the two spectral sequences from
\(E_1\) onward.  Its construction accounts for the stated dependence
on choices, and the final assertion follows immediately.
\end{proof}

\begin{corollary}[Orbifold compactification independence]
\label{cor:orbifold-irregular-hodge-independence}
Let \((\mathscr U,w)\) be a smooth separated Deligne--Mumford
Landau--Ginzburg model, and write
\(I\mathscr U=\coprod_{\nu\in\Lambda}\mathscr U_\nu\), with ages
\(a_\nu\).  For every sector compactification datum
\(\boldsymbol{\mathscr X}\) and all \(q,\lambda\in\Q\),
one has the literal equality
\begin{equation}
 F_{\mathrm{orb},\boldsymbol{\mathscr X}}^\lambda
 H_{\mathrm{dR,orb}}^q(\mathscr U,w)
 =
 \bigoplus_{\substack{\nu\in\Lambda\\q-2a_\nu\in\Z}}
 F_{\mathrm{irr}}^{\lambda-a_\nu}
 H_{\mathrm{dR}}^{q-2a_\nu}(\mathscr U_\nu,w_\nu).
 \label{eq:intrinsic-orbifold-irregular-hodge-filtration}
\end{equation}
In particular, the left-hand side is independent of
\(\boldsymbol{\mathscr X}\); denote it by
\(F_{\mathrm{orb}}^\lambda
H_{\mathrm{dR,orb}}^q(\mathscr U,w)\).  Consequently
\begin{equation}
 f_{\mathrm{orb}}^{\lambda,\mu}(\mathscr U,w)
 :=\dim\Gr_{F_{\mathrm{orb}}}^{\lambda}
 H_{\mathrm{dR,orb}}^{\lambda+\mu}(\mathscr U,w)
 =f_{\mathrm{orb},\boldsymbol{\mathscr X}}^{\lambda,\mu}
 (\mathscr U,w)
 \label{eq:intrinsic-orbifold-irregular-hodge-number}
\end{equation}
is independent of all sector compactifications for every
\(\lambda,\mu\in\Q\).

If \((X,D,w)\) is a nondegenerate projective toroidal orbifold
compactification in the sense of Harder--Lee, the comparison in
\cref{thm:harder-lee-orbifold-comparison} identifies their filtration
and irregular orbifold Hodge numbers with
\eqref{eq:intrinsic-orbifold-irregular-hodge-filtration} and
\eqref{eq:intrinsic-orbifold-irregular-hodge-number}.  Hence these
objects depend only on \((\mathscr U,w)\), not on the chosen
Harder--Lee compactification.
\end{corollary}

\begin{proof}
Apply
\cref{thm:introduction-nc-rational-compactification-independence}
to each sector \((\mathscr U_\nu,w_\nu)\).  Its equality of image
filtrations identifies every summand in the modelwise definition with
the corresponding \(F_{\mathrm{irr}}\)-summand.  The inertia stack
has finitely many connected components, so taking their direct sum
proves \eqref{eq:intrinsic-orbifold-irregular-hodge-filtration}.
Finite direct sums commute with the unions that define
\(F^{>\lambda}\).  Taking associated graded spaces and dimensions
therefore proves
\eqref{eq:intrinsic-orbifold-irregular-hodge-number}.  The final
assertion follows from the fixed-compactification comparison theorem.
\end{proof}
 
\bibliographystyle{amsalpha}
\newcommand{\etalchar}[1]{$^{#1}$}
\providecommand{\bysame}{\leavevmode\hbox to3em{\hrulefill}\thinspace}
\providecommand{\MR}{\relax\ifhmode\unskip\space\fi MR }
\providecommand{\MRhref}[2]{\href{http://www.ams.org/mathscinet-getitem?mr=#1}{#2}
}
\providecommand{\href}[2]{#2}

\end{document}